\documentclass[12pt]{amsart}
\usepackage{amscd,amsmath,amssymb,enumerate,amsfonts,mathrsfs,txfonts}
\usepackage{comment}
\usepackage{hyperref}
\usepackage{xcolor}
\usepackage{tikz,pgfplots,pgf}
\usepackage[all]{xy}

\tikzset{node distance=3cm, auto}
\usetikzlibrary{calc} 
\usetikzlibrary{trees} 

\usepackage{vmargin}
\setpapersize{A4}
\setmargins{2.5cm}      
{1.5cm}                        
{16.5cm}                      
{23.42cm}                   
{10pt}                          
{1cm}                           
{0pt}                          
{2cm}

\usepackage{tikz,pgfplots,pgf}
\usepackage[all]{xy}
\tikzset{node distance=3cm, auto}
\usetikzlibrary{calc} 
\usetikzlibrary{trees} 

\newtheorem{theorem}{Theorem}[section]         
             
\newtheorem{corollary}[theorem]{Corollary}
\newtheorem{proposition}[theorem]{Proposition}
\newtheorem{definition}[theorem]{Definition}

\newtheorem{remark}{Remark}

\def\N{\mathbb{N}}
\def\R{\mathbb{R}}
\def\C{\mathbb{C}}
\def\K{\mathcal{K}}

\def\G{\mathcal{G}}
\def\H{\mathcal{H}}
\def\I{\mathcal{I}}
\def\L{\mathcal{L}}

\def\P{\mathcal{P}}
\def\W{\mathcal{W}}

\def\lin{\mathrm{lin}}

\def\Hv{\H_v^{\infty}}

\def\Gv{\G_{v}}

\begin{document}

\title[On $p$-summability in weighted holomorphic function spaces]{On $p$-summability in weighted Banach spaces \\ of holomorphic functions}

\author[M. G. Cabrera-Padilla]{M. G. Cabrera-Padilla}
\address{Departamento de Matem\'{a}ticas\\ Universidad de Almer\'{i}a\\ 04120 Almer\'{i}a\\ Spain}
\email{m\_gador@hotmail.com}
\thanks{Corresponding author: A. Jim{\'e}nez-Vargas.} 

\author[A. Jim{\'e}nez-Vargas]{A. Jim\'enez-Vargas}
\address{Departamento de Matem\'aticas\\ Universidad de Almer\'ia\\ 04120 Almer\'ia\\ Spain}
\email{ajimenez@ual.es}

\author[A. Keten \c Copur]{A. Keten \c Copur}
\address{Department of Mathematics and Computer Science\\ Faculty of Sciences\\ Necmettin Erbakan University\\ 42090 Konya\\ T\"{u}rkiye}
\email{aketen@erbakan.edu.tr}



\date{December 10, 2025}

\begin{abstract}
Given an open subset $U$ of a complex Banach space $E$, a weight $v$ on $U$, and a complex Banach space $F$, let $\Hv(U,F)$ denote the Banach space of all weighted holomorphic mappings $f\colon U\to F$, under the weighted supremum norm $\left\|f\right\|_v:=\sup\left\{v(x)\left\|f(x)\right\|\colon x\in U\right\}$. In this paper, we introduce and study the class $\Pi_p^{\Hv}(U,F)$ of $p$-summing weighted holomorphic mappings. We prove that it is an injective Banach ideal of weighted holomorphic mappings. Variants for weighted holomorphic mappings of Pietsch Domination Theorem, Pietsch Factorization Theorem and Maurey Extrapolation Theorem are presented. We also identify the spaces of $p$-summing weighted holomorphic mappings from $U$ into $F^*$ under the norm $\pi^{\Hv}_p$ with the duals of $F$-valued $\Hv$-molecules on $U$ under a suitable version $d^{\Hv}_{p^*}$ of the Chevet--Saphar tensor norms.
\end{abstract}

\maketitle


\section*{Introduction}

The notion of $p$-summing operators between Banach spaces, introduced by Grothendieck \cite{Gro-55} for $p=1$ and by Pietsch \cite{Pie-67} for any $p>0$, has played an important role in the theory of Banach spaces (see the monographs \cite{Defant-Floret,DisJarTon-95,Pie-80,Rya-02}). 

Its influence has motivated the study of the theory of $p$-summability for different classes of functions in addition to continuous linear operators. This is the case of the research on multilinear operators and polynomials initiated by Pietsch \cite{Pie-84} and continued by Achour and Mezrag \cite{AchMez-07}, Angulo-L\'opez and Fern\'andez-Unzueta \cite{AngFer-20} and Dimant \cite{Dim-03} (see also the survey \cite{PelRuSan-16} by Pellegrino, Rueda and S\'anchez-P\'erez about the $p$-summability in this context); Lipschitz mappings addressed by Farmer and Johnson \cite{FarJoh-09} and Saadi \cite{Saa-15}; and holomorphic mappings studied by Matos \cite{Mat-96} and Pellegrino \cite{Pel-03}.

Let $E$ and $F$ be complex Banach spaces and let $U$ be an open subset of $E$. Let $\H(U,F)$ be the space of all holomorphic mappings from $U$ into $F$. A weight $v$ on $U$ is a (strictly) positive continuous function. The space of weighted holomorphic mappings, denoted by $\Hv(U,F)$, is the Banach space of all mappings $f\in\H(U,F)$ such that   
$$
\left\|f\right\|_v:=\sup\left\{v(x)\left\|f(x)\right\|\colon x\in U\right\}<\infty ,
$$ 
under the weighted supremum norm $\|\cdot\|_v$. In particular, $\Hv(U):=\Hv(U,\C)$. 

Weighted Banach spaces $\Hv(U,F)$ appear in the study of growth conditions of analytic functions and have been addressed in many papers on Complex Analysis, Spectral Theory, Fourier Analysis, Partial Differential Equations and Convolution Equations. We refer to the survey by Bonet \cite{Bon-22} and the references therein for a complete information on weighted analytic functions defined on domains of finite dimension. The study of weighted spaces of
holomorphic functions on infinite-dimensional domains was addressed in \cite{GarMaeRue-00,GarMaeSev-04,GupBaw-16,Jord-13}. A wide range of weights $v$ on $U=\mathbb{D},\mathbb{C}$ that give rise to concrete and outstanding $\Hv(U)$-spaces can be found in \cite{Bon-22}. Moreover, certain conditions on some of these weights are equivalent to the fact that $\Hv(U)$ contains the polynomials. At this point, it is also important to note that there exist various definitions of $p$-summability in the polynomial setting (see \cite{Pel-03,PelRuSan-16}), and some of those constructions are closely related to ours.

We will develop in this paper a theory of $p$-summability for weighted holomorphic mappings. Thus, we introduce the notion of $p$-summing weighted holomorphic mappings and show that it is appropriate to establish variants for the weighted holomorphic setting of some known results of the theory of $p$-summing operators. 

Thus, Section \ref{2} shows that $p$-summing weighted holomorphic mappings enjoy an injective Banach ideal property and satisfy Pietsch Domination Theorem, Factorization Theorem through spaces of continuous functions, Maurey Extrapolation Theorem, and Inclusion and Coincidence Theorem. 

Section \ref{3} contains a tensorial approach to this new class of weighted holomorphic mappings. In this way, we introduce a suitable tensor product space equipped with a Chevet--Saphar type tensor norm in such a way that its dual is isometrically isomorphic to the space of $p$-summing weighted holomorphic mappings.

We will need to fix some notation and recall some facts on weighted holomorphic maps. \\

\textbf{Notation.} As usual, $\mathcal{L}(E,F)$ denotes the Banach space of all bounded linear operators from $E$ into $F$, equipped with the operator canonical norm. In particular, $\mathcal{L}(E,\mathbb{C})$ is denoted by $E^*$. For $x\in E$ and $x^*\in E^*$, we will sometimes write $\langle x^*,x\rangle$ instead of $x^*(x)$. $B_E$ stands for the closed unit ball of $E$. We will write $(E,\|\cdot\|_E)\leq(F,\|\cdot\|_F)$ to indicate that $E\subseteq F$ and $\left\|x\right\|_F\leq\left\|x\right\|_E$ for all $x\in E$. For any $1<p<\infty$, let $p^*=p/(p-1)$ denote the conjugate index of $p$. \\ 


\section{$p$-Summing weighted holomorphic mappings}\label{2}

For any $p\in [1,\infty)$ and Banach spaces $E,F$, an operator $T\in\L(E,F)$ is said to be $p$-summing if there exists a constant $C>0$ such that
$$
\left(\sum_{i=1}^{n}\left\Vert T(x_{i})\right\Vert^{p}\right)^{\frac{1}{p}}
\leq C\sup_{x^*\in B_{E^*}}\left(\sum_{i=1}^{n}\left\vert x^*\left(x_{i}\right)\right\vert^{p}\right)^{\frac{1}{p}},
$$
for all $n\in\mathbb{N}$ and $(x_i)_{i=1}^{n}\subseteq E$. The smallest constant $C$ satisfying the inequality above is denoted by $\pi_p(T)$, and the linear space of all $p$-summing operators from $E$ into $F$ by $\Pi_p(E,F)$. Moreover, $[\Pi_p,\pi_p]$ is a Banach operator ideal. 

We begin by introducing the following variant for weighted holomorphic mappings of the concept of $p$-summing linear operators between Banach spaces. 

From now on, unless otherwise stated, we will suppose that $p\in [1,\infty]$ and $E$, $F$ and $G$ will denote complex Banach spaces, $U$ will be an open subset of $E$, and $v$ will be a weight on $U$. To simplify the notation, when we talk about a weighted holomorphic function, we will understand that it refers with respect to the weight $v$.

\begin{definition}\label{def-p-summing weighted holomorphic}
Given $1\leq p\leq\infty$, a mapping $f\in\H(U,F)$ is said to be $p$-summing weighted holomorphic if there is a constant $C\geq 0$ such that 
$$
\left(\sum_{i=1}^n\left|\lambda_i\right|^pv(x_i)^p\left\|f(x_i)\right\|^p\right)^{\frac{1}{p}}\leq C \sup_{g\in B_{\Hv(U)}}\left(\sum_{i=1}^n\left|\lambda_i\right|^p v(x_i)^p\left|g(x_i)\right|^p\right)^{\frac{1}{p}}
$$
if $1\leq p<\infty$, and 
$$
\max_{1\leq i\leq n}\left|\lambda_i\right|v(x_i)\left\|f(x_i)\right\|\leq C \sup_{g\in B_{\Hv(U)}}\left(\max_{1\leq i\leq n}\left|\lambda_i\right|v(x_i)\left|g(x_i)\right|\right)
$$
if $p=\infty$, for any $n\in\mathbb{N}$, $\lambda_1,\ldots,\lambda_n\in\mathbb{C}$ and $x_1,\ldots,x_n\in U$. We denote by $\pi^{\Hv}_p(f)$ the infimum of all constants $C$ satisfying the inequality above, and by $\Pi^{\Hv}_p(U,F)$ the set of all $p$-summing weighted holomorphic mappings from $U$ into $F$. 
\end{definition}

We now show that every $p$-summing weighted holomorphic mapping is indeed a weighted holomorphic mapping and this justifies the terminology used in Definition \ref{def-p-summing weighted holomorphic}. In addition, we make a comment on the weights $v$ to avoid trivial cases.

\begin{remark}\label{rem-p-summing weighted holomorphic}
For $p\in [1,\infty)$, $(\Pi^{\Hv}_p(U,F),\pi^{\Hv}_p)\leq(\Hv(U,F),\|\cdot\|_v)$. Indeed, given $f\in\Pi^{\Hv}_p(U,F)$, for all $x\in U$, note that    
$$
v(x)\|f(x)\|=(v(x)^p\|f(x)\|^p)^{\frac{1}{p}}\leq\pi_p^{\Hv}(f)\sup_{g\in B_{\Hv(U)}}(v(x)^p\left|g(x)\right|^p)^{\frac{1}{p}}\leq\pi_p^{\Hv}(f).
$$
It is worth recalling that $(\Pi^{\Hv}_\infty(U,F),\pi^{\Hv}_\infty)=(\Hv(U,F),\|\cdot\|_v)$ (see Proposition \ref{prop-1}). In order to avoid the elementary case where the space $\Pi^{\Hv}_p(U,F)$ coincides with $\Hv(U,F)$, we could impose on the weight $v$ some of the conditions: $\sup_{x\in U}v(x)=\infty$ or $\inf_{x\in U}v(x)=0$. Indeed, otherwise, note that there are $N > 0$ and $M > 0$ such that $N<v (x)< M$ for all $x\in U$, hence $N\|\cdot\|_\infty\leq\|\cdot\|_v\leq M\|\cdot\|_\infty$, and thus $\Hv(U,F)=\H^\infty(U,F)$ (the space of bounded holomorphic mappings from $U$ into $F$) with equivalent norms. Moreover, the constant function $g(x)=1/M$ for all $x\in U$,  belongs to $B_{\Hv(U)}$, and then every $f\in\Hv(U,F)$ belongs to $\Pi^{\Hv}_p(U,F)$ since 
\begin{align*}
\sum_{i=1}^n|\lambda_i|^pv(x_i)^p\|f(x_i)\|^p
&\leq M^p\|f\|^p_\infty\sum_{i=1}^n|\lambda_i|^pv(x_i)^p\left(\frac{1}{M}\right)^p\\
&\leq M^p\|f\|^p_\infty\sup_{g\in B_{\Hv(U)}}\sum_{i=1}^n|\lambda_i|^pv(x_i)^p|g(x_i)|^p,
\end{align*}
for any $n\in\mathbb{N}$, $\lambda_1,\ldots,\lambda_n\in\mathbb{C}$ and $x_1,\ldots,x_n\in U$. 
\end{remark}

Our main purpose in this paper is to study the spaces of Banach-valued $p$-summing weighted holomorphic mappings. We will go presenting their main properties in different subsections. 


\subsection{Linearization}

Following \cite{BieSum-93,BonDomLin-01,GupBaw-16}, $\Gv(U)$ is the space of all linear functionals on $\Hv(U)$ whose restriction to $B_{\Hv(U)}$ is continuous for the compact-open topology. 

The following result gathers some facts concerning the linearization of weighted holomorphic mappings that will applied through the paper.

\begin{theorem}\label{t0}\cite{BieSum-93,BonDomLin-01,GupBaw-16} 
Let $U$ be an open set of a complex Banach space $E$ and let $v$ be a weight on $U$.
\begin{enumerate}
\item[(i)]\label{ap(i)} $\Gv(U)$ is a closed subspace of $\Hv(U)^*$, and the mapping $J_v\colon\Hv(U)\to\Gv(U)^*$, given by $J_v(g)(\phi)=\phi(g)$ for $\phi\in\Gv(U)$ and $g\in\Hv(U)$, is an isometric isomorphism.
\item[(ii)]\label{ap(ii)} For each $x\in U$, the functional $\delta_x\colon\Hv(U)\to\mathbb{C}$, defined by $\delta_x(g)=g(x)$ for $g\in\Hv(U)$, is in $\Gv(U)$, and there exists $g_x\in B_{\Hv(U)}$ such that $g_x(x)=\left\|\delta_x\right\|:=\sup_{g\in B_{\Hv(U)}}\left|g(x)\right|$.
\item[(iii)]\label{ap(iii)} The mapping $\Delta_v\colon U\to\Gv(U)$ given by $\Delta_v(x)=\delta_x$ is in $\Hv(U,\Gv(U))$ with $\left\|\Delta_v\right\|_v\leq 1$.
\item[(iv)]\label{ap(vi)} For every complex Banach space $F$ and every mapping $f\in\Hv(U,F)$, there exists a unique operator $T_f\in\L(\Gv(U),F)$ such that $T_f\circ\Delta_v=f$. Furthermore, $\left\|T_f\right\|=\left\|f\right\|_v$. 
$\hfill\qed$
\end{enumerate}
\end{theorem}

We establish the following relation of the $p$-summability of a weighted holomorphic mapping and its linearization.

\begin{theorem}\label{t1.2}
Let $f\in\Hv(U,F)$ and $p\in [1,\infty)$. If the linear operator $T_f\colon\Gv(U)\to F$ is $p$-summing, then $f$ is $p$-summing weighted holomorphic and $\pi_p^{\Hv}(f)\leq\pi_p(T_f)$.
\end{theorem}

\begin{proof}
If $T_f\in\Pi_p(\Gv(U),F)$, then 
$$
\left(\sum_{i=1}^n\left\|T_f(\phi_i)\right\|^p\right)^{\frac{1}{p}}
\leq \pi_p(T_f)\sup_{S\in B_{\Gv(U)^*}}\left(\sum_{i=1}^n\left|S(\phi_i)\right|^p\right)^{\frac{1}{p}}
$$
for any $n$ in $\N$ and $\phi_1,\ldots,\phi_n$ in $\Gv(U)$. In particular, we have
\begin{align*}
\left(\sum_{i=1}^n\left|\lambda_i\right|^pv(x_i)^p\left\|f(x_i)\right\|^p\right)^{\frac{1}{p}}
&=\left(\sum_{i=1}^n\left\|T_f(\lambda_iv(x_i)\delta_{x_i})\right\|^p\right)^{\frac{1}{p}}\\
&\leq\pi_p(T_f)\sup_{S\in B_{\Gv(U)^*}}\left(\sum_{i=1}^n\left|S(\lambda_iv(x_i)\delta_{x_i})\right|^p\right)^{\frac{1}{p}}\\
&=\pi_p(T_f)\sup_{g\in B_{\Hv(U)}}\left(\sum_{i=1}^n\left|\lambda_i\right|^pv(x_i)^p\left|J_v(g)(\delta_{x_i})\right|^p\right)^{\frac{1}{p}}\\
&=\pi_p(T_f)\sup_{g\in B_{\Hv(U)}}\left(\sum_{i=1}^n\left|\lambda_i\right|^pv(x_i)^p\left|g(x_i)\right|^p\right)^{\frac{1}{p}}
\end{align*}
for any $n\in\mathbb{N}$, $\lambda_1,\ldots,\lambda_n\in\mathbb{C}$ and $x_1,\ldots,x_n\in U$. Consequently, $f\in\Pi^{\Hv}_p(U,F)$ and $\pi_p^{\Hv}(f)\leq\pi_p(T_f)$.
\end{proof}

From Theorem \ref{t1.2} we deduce a result that will be applied later.

\begin{corollary}\label{new corollary}
Let $p\in [1,\infty)$. If $f\colon U\to F$ is a weighted holomorphic mapping and $S\colon F\to G$ is a $p$-summing linear operator, then $S\circ f\colon U\to G$ is $p$-summing weighted holomorphic and $\pi^{\Hv}_p(S\circ f)\leq\pi_p(S)\|f\|_v$.
\end{corollary}

\begin{proof}
Assume that $f\in\Hv(U,F)$ and $S\in\Pi_p(F,G)$. Since $T_{S\circ f}=S\circ T_f$ with $\left\|T_f\right\|=\|f\|_v$ by Theorem \ref{t0}, then $T_{S\circ f}\in\Pi_p(\Gv(U),G)$ with $\pi_p(T_{S\circ f})\leq\pi_p(S)\left\|T_f\right\|$ by the ideal property of $[\Pi_p,\pi_p]$ (see \cite[Theorem 2.4]{DisJarTon-95}). Hence $S\circ f\in\Pi_p^{\Hv}(U,G)$ and $\pi^{\Hv}_p(S\circ f)\leq\pi_p(T_{S\circ f})\leq\pi_p(S)\|f\|_v$ by Theorem \ref{t1.2}.
\end{proof}

\subsection{Injective Banach ideal property}

We first prove that the set of $p$-summing weighted holomorphic mappings enjoys an injective Banach ideal structure. 

Let us recall from \cite{CabJimAys-24} that a Banach ideal of weighted holomorphic mappings (in short, a Banach weighted holomorphic ideal) is an assignment $\left[\I^{\Hv},\|\cdot\|_{\I^{\Hv}}\right]$ which associates with every pair $(U,F)$, where $U$ is an open subset of a complex Banach space $E$ and $F$ is a complex Banach space, a set $\I^{\Hv}(U,F)\subseteq\Hv(U,F)$ and a function $\|\cdot\|_{\I^{\Hv}}\colon \I^{\Hv}(U,F)\to\mathbb{R}^+_0$ satisfying the properties:
\begin{itemize}
\item[(P1)] $\left(\I^{\Hv}(U,F),\|\cdot\|_{\I^{\Hv}}\right)$ is a Banach space with $\|f\|_{\I^{\Hv}}\geq \|f\|_v$ for all $f\in\I^{\Hv}(U,F)$.
\item[(P2)] For any $h\in\Hv(U)$ and $y\in F$, the map $h\cdot y\colon x\mapsto h(x)y$ from $U$ to $F$ is in $\I^{\Hv}(U,F)$ with $\|h\cdot y\|_{\I^{\Hv}}=\|h\|_v\|y\|$.
\item[(P3)] The ideal property: if $V$ is an open subset of $E$ such that $V\subseteq U$, $h\in\mathcal{H}(V,U)$ with $c_v(h):=\sup_{x\in V}(v(x)/v(h(x)))<\infty$
, $f\in \I^{\Hv}(U,F)$ and $T\in\L(F,G)$, where $G$ is a complex Banach space, then $T\circ f\circ h\in\I^{\Hv}(V,G)$ with $\|T\circ f\circ h\|_{\I^{\Hv}}\leq \|T\|\|f\|_{\I^{\Hv}}c_v(h)$. 
\end{itemize}

A Banach weighted holomorphic ideal $[\I^{\Hv},\|\cdot\|_{\I^{\Hv}}]$ is called:
\begin{itemize}
\item[(I)] Injective if for any mapping $f\in\Hv(U,F)$, any complex Banach space $G$ and any isometric linear embedding $\iota\colon F\to G$, we have that $f\in\I^{\Hv}(U,F)$ with $\left\|f\right\|_{\I^{\Hv}}=\left\|\iota\circ f\right\|_{\I^{\Hv}}$ whenever $\iota\circ f\in\I^{\Hv}(U,G)$.
\end{itemize}

\begin{proposition}\label{ideal summing} 
$[\Pi^{\Hv}_p,\pi^{\Hv}_p]$ is an injective Banach weighted holomorphic ideal.
\end{proposition}

\begin{proof}
We will prove the case $1\leq p<\infty$, and for $p=\infty$ a similar proof works.

(P1) Let $n\in\mathbb{N}$, $\lambda_1,\ldots,\lambda_n\in\mathbb{C}$ and $x_1,\ldots,x_n\in U$. If $f\in\Pi^{\Hv}_p(U,F)$, then $\|f\|_v\leq\pi^{\Hv}_p(f)$ by Remark \ref{rem-p-summing weighted holomorphic}. If $\pi^{\Hv}_p(f)=0$, then $\|f\|_v=0$ and so $f=0$. Given $f_1,f_2\in\Pi^{\Hv}_p(U,F)$, one has 
\begin{align*}
\left(\sum_{i=1}^n\left|\lambda_i\right|^pv(x_i)^p\left\|(f_1+f_2)(x_i)\right\|^p\right)^{\frac{1}{p}}
&\leq\left(\sum_{i=1}^n\left(\left|\lambda_i\right|v(x_i)\left\|f_1(x_i)\right\|+\left|\lambda_i\right|v(x_i)\left\|f_2(x_i)\right\|\right)^p\right)^{\frac{1}{p}}\\
&\leq  \left(\sum_{i=1}^n \left|\lambda_i\right|^p v(x_i)^p \|f_1(x_i)\|^p\right)^{\frac{1}{p}}+ \left(\sum_{i=1}^n \left|\lambda_i\right|^p v(x_i)^p \|f_2(x_i)\|^p\right)^{\frac{1}{p}}\\
&\leq \left(\pi_p^{\Hv}(f_1)+\pi_p^{\Hv}(f_2)\right)\sup_{g\in B_{\Hv(U)}}\left(\sum_{i=1}^n \left|\lambda_i\right|^p v(x_i)^p \left|g(x_i)\right|^p\right)^{\frac{1}{p}},
\end{align*}
and thus $f_1+f_2\in\Pi^{\Hv}_p(U,F)$ with $\pi^{\Hv}_p(f_1+f_2)\leq\pi^{\Hv}_p(f_1)+\pi^{\Hv}_p(f_2)$.

Let $\lambda\in\mathbb{C}$ and $f\in\Pi^{\Hv}_p(U,F)$. An easy calculation yields 
\begin{align*}
\left(\sum_{i=1}^n\left|\lambda_i\right|^pv(x_i)^p\left\|(\lambda f)(x_i)\right\|^p\right)^{\frac{1}{p}}
&=\left|\lambda\right|\left(\sum_{i=1}^n\left|\lambda_i\right|^p v(x_i)^p\left\|f(x_i)\right\|^p\right)^{\frac{1}{p}}\\
&\leq\left|\lambda\right|\pi^{\Hv}_p(f)\sup_{g\in B_{\Hv(U)}}\left(\sum_{i=1}^n\left|\lambda_i\right|^pv(x_i)^p\left|g(x_i)\right|^p\right)^{\frac{1}{p}},
\end{align*}
and so $\lambda f\in\Pi^{\Hv}_p(U,F)$ with $\pi^{\Hv}_p(\lambda f)\leq|\lambda|\pi^{\Hv}_p(f)$. Hence $\pi^{\Hv}_p(\lambda f)=0=\left|\lambda\right|\pi^{\Hv}_p(f)$ if $\lambda=0$. For $\lambda\neq 0$, it is clear that $\pi^{\Hv}_p(f)=\pi^{\Hv}_p(\lambda^{-1}(\lambda f))\leq\left|\lambda\right|^{-1}\pi^{\Hv}_p(\lambda f)$, hence $\left|\lambda\right|\pi^{\Hv}_p(f)\leq\pi^{\Hv}_p(\lambda f)$, and, consequently, $\pi^{\Hv}_p(\lambda f)=\left|\lambda\right|\pi^{\Hv}_p(f)$. In this way, $\left(\Pi^{\Hv}_p(U,F),\pi^{\Hv}_p\right)$ is a normed space.

To show that it is a Banach space, let $(f_n)_{n\geq 1}$ be a sequence in $\Pi^{\Hv}_p(U,F)$ such that $\sum_{n\geq 1}\pi^{\Hv}_p(f_n)$ converges. Since $\|f_n\|_v\leq\pi^{\Hv}_p(f_n)$ for all $n\in\mathbb{N}$ and $\left(\Hv(U,F),\|\cdot\|_v\right)$ is a Banach space, then the series $\sum_{n\geq 1} f_n$ converges in $\left(\Hv(U,F),\|\cdot\|_v\right)$ to a function $f\in\Hv(U,F)$. This implies that $\sum_{n\geq 1} f_n$ converges pointwise on $U$ to $f$ because 
$$
v(x)\left\|\sum_{i=1}^nf_i(x)-f(x)\right\|\leq \left\|\sum_{i=1}^nf_i-f\right\|_v
$$
for each $x\in U$ and for all $n\in\mathbb{N}$. Given $m\in\mathbb{N}$, $x_1,\ldots,x_m\in U$ and $\lambda_1,\ldots,\lambda_m\in\mathbb{C}$, we have 
\begin{align*}
\left(\sum_{k=1}^m\left|\lambda_k\right|^pv(x_k)^p\left\|\sum_{i=1}^nf_i(x_k)\right\|^p\right)^{\frac{1}{p}}
&\leq \pi^{\Hv}_p\left(\sum_{i=1}^nf_i\right)\sup_{g\in B_{\Hv(U)}}\left(\sum_{k=1}^m\left|\lambda_k\right|^pv(x_k)^p\left|g(x_k)\right|^p\right)^{\frac{1}{p}}\\
&\leq \sum_{i=1}^n\pi^{\Hv}_p(f_i)\sup_{g\in B_{\Hv(U)}}\left(\sum_{k=1}^m\left|\lambda_k\right|^pv(x_k)^p\left|g(x_k)\right|^p\right)^{\frac{1}{p}}\\
\end{align*}
for all $n\in\mathbb{N}$. Taking limits with $n\to +\infty$, we deduce   
$$
\left(\sum_{k=1}^m\left|\lambda_k\right|^pv(x_k)^p\left\|f(x_k)\right\|^p\right)^{\frac{1}{p}}
\leq \sum_{i=1}^\infty\pi^{\Hv}_p(f_i)\sup_{g\in B_{\Hv(U)}}\left(\sum_{k=1}^m\left|\lambda_k\right|^pv(x_k)^p\left|g(x_k)\right|^p\right)^{\frac{1}{p}}.
$$
Hence $f\in\Pi^{\Hv}_p(U,F)$ with $\pi^{\Hv}_p(f)\leq\sum_{n=1}^\infty \pi^{\Hv}_p(f_n)$. Furthermore,  
$$
\pi^{\Hv}_p\left(f-\sum_{i=1}^nf_i\right)=\pi^{\Hv}_p\left(\sum_{i=n+1}^\infty f_i\right)\leq\sum_{i=n+1}^\infty \pi^{\Hv}_p(f_i) 
$$
for all $n\in\mathbb{N}$, and so $f$ is the $\pi^{\Hv}_p$-limit of the series $\sum_{n\geq 1} f_n$. This proves the completeness of $\pi^{\Hv}_p$.

(P2) Let $h\in\Hv(U)$ and $y\in F$. Note that $h\cdot y\in\Hv(U,F)$ with $\|h\cdot y\|_v=\|h\|_v\left\|y\right\|$. We can suppose $h\neq 0$ and obtain 
\begin{align*}
\left(\sum_{i=1}^n\left|\lambda_i\right|^pv(x_i)^p\left\|(h\cdot y)(x_i)\right\|^p\right)^{\frac{1}{p}}
&=\left\|y\right\|\|h\|_v\left(\sum_{i=1}^n\left|\lambda_i\right|^pv(x_i)^p\left|\left(\frac{h}{\|h\|_v}\right)(x_i)\right|^p\right)^{\frac{1}{p}}\\
&\leq\left\|y\right\|\|h\|_v\sup_{g\in B_{\Hv(U)}}\left(\sum_{i=1}^n\left|\lambda_i\right|^pv(x_i)^p\left|g(x_i)\right|^p\right)^{\frac{1}{p}}.
\end{align*}
Thus $h\cdot y\in\Pi^{\Hv}_p(U,F)$ with $\pi^{\Hv}_p(h\cdot y)\leq \|h\|_v\left\|y\right\|$. Conversely, note that $\|h\|_v\|y\|\leq\pi^{\Hv}_p(h\cdot y)$ by Remark \ref{rem-p-summing weighted holomorphic}.

(P3) Let $V$ be an open subset of $E$ such that $V\subseteq U$, $h\in \mathcal{H}(V,U)$ with $c_v(h)<\infty$, $f\in\Pi^{\Hv}_p(U,F)$ and $T\in\L(F,G)$. Observe that $T\circ f\circ h\in\Hv(V,G)$. For any $n\in\mathbb{N}$, $\lambda_1,\ldots,\lambda_n\in\mathbb{C}$ and $x_1,\ldots,x_n\in V$, we obtain 
\begin{align*}
&\left(\sum_{i=1}^n\left|\lambda_i\right|^pv(x_i)^p\left\|(T\circ f\circ h)(x_i)\right\|^p\right)^{\frac{1}{p}}\\
&\leq\|T\|\left(\sum_{i=1}^n |\lambda_i|^p \frac{v(x_i)^p}{v(h(x_i))^p}v(h(x_i))^p \|f(h(x_i))\|^p \right)^{\frac{1}{p}}\\
&\leq \|T\|\pi_p^{\Hv}(f)\sup_{g\in B_{\Hv(U)}}\left(\sum_{i=1}^n |\lambda_i|^p \frac{v(x_i)^p}{v(h(x_i))^p}v(h(x_i))^p |g(h(x_i))|^p \right)^{\frac{1}{p}} \\ 
&=\|T\|\pi_p^{\Hv}(f)\sup_{g\in B_{\Hv(U)}}\left(\sum_{i=1}^n |\lambda_i|^p v(x_i)^p|(g\circ h)(x_i)|^p \right)^{\frac{1}{p}}\\ 
&\leq\|T\|\pi_p^{\Hv}(f)c_v(h)\sup_{0\neq g\in B_{\Hv(U)}}\left(\sum_{i=1}^n |\lambda_i|^p v(x_i)^p||g||^p_v\left|\left(\frac{g\circ h}{c_v(h)||g||_v}\right)(x_i)\right|^p \right)^{\frac{1}{p}}\\ 
&\leq \|T\|\pi_p^{\Hv}(f)c_v(h)\sup_{g_0\in B_{\Hv(V)}}\left(\sum_{i=1}^n |\lambda_i|^p v(x_i)^p|g_0(x_i)|^p \right)^{\frac{1}{p}},
\end{align*}
since $g\circ h\in\Hv(V)$ with $\|g\circ h\|_v\leq c_v(h)\|g\|_v$. So, $T\circ f\circ h\in\Pi^{\Hv}_p(V,G)$ and $\pi^{\Hv}_p(T\circ f\circ h)\leq\left\|T\right\|\pi^{\Hv}_p(f)c_v(h)$.

(I) Let $f\in\Hv(U,F)$ and $\iota\colon F\to G$ be an into linear isometry such that $\iota\circ f\in\Pi^{\Hv}_p(U,G)$. Given $n\in\mathbb{N}$, $\lambda_1,\ldots,\lambda_n\in\mathbb{C}$ and $x_1,\ldots,x_n\in U$, we have 
\begin{align*}
\left(\sum_{i=1}^n\left|\lambda_i\right|^pv(x_i)^p\left\|f(x_i)\right\|^p\right)^{\frac{1}{p}}
&=\left(\sum_{i=1}^n\left|\lambda_i\right|^pv(x_i)^p\left\|(\iota\circ f)(x_i))\right\|^p\right)^{\frac{1}{p}}\\
&\leq \pi^{\Hv}_p(\iota\circ f)\sup_{g\in B_{\Hv(U)}}\left(\sum_{i=1}^n\left|\lambda_i\right|^pv(x_i)^p\left|g(x_i)\right|^p\right)^{\frac{1}{p}},
\end{align*}
and so $f\in\Pi^{\Hv}_p(U,F)$ with $\pi^{\Hv}_p(f)\leq\pi^{\Hv}_p(\iota\circ f)$. The opposite inequality follows from (P3).
\end{proof}


\subsection{Pietsch domination}

We can characterize $p$-summing weighted holomorphic mappings in terms of the known property of Pietsch domination (see \cite[Theorem 2]{Pie-67}). 

In light of Theorem \ref{t0}, $\Hv(U)$ is a dual Banach space and, consequently, this space can be endowed with the weak* topology. Let $\P(B_{\Hv(U)})$ be the set of all Borel regular probability measures $\mu$ on the Hausdorff compact set $(B_{\Hv(U)},w^*)$.

\begin{theorem}\label{Pietsch} 
Let $1\leq p<\infty$ and $f\in\Hv(U,F)$. The following statements are equivalent:
\begin{enumerate}
\item[(i)] $f$ is $p$-summing weighted holomorphic.
\item[(ii)] There is a constant $C\geq 0$ and a measure $\mu\in\P(B_{\Hv(U)})$ such that 
$$
\left\|f(x)\right\|\leq C\left(\int_{B_{\Hv(U)}}\left|g(x)\right|^{p}d\mu(g)\right)^{\frac{1}{p}}\qquad (x\in U).
$$
\end{enumerate}
In this case, $\pi^{\Hv}_p(f)$ is the infimum of all constants $C\geq0$ satisfying the inequality in $(ii)$ and, in fact, this infimum is attained.
\end{theorem}

\begin{proof}
$(i)\Rightarrow (ii)$: In order to apply an unified abstract version of Piestch Domination Theorem stated in \cite{PelSan-11}, consider the functions $S\colon\Hv(U,F)\times U\times\C\to [0,\infty[$ and $R\colon B_{\Hv(U)}\times U\times\C\to [0,\infty[$  defined, respectively, by $S(f,x,\lambda)=\left|\lambda\right|v(x)\left\|f(x)\right\|$ and $R(g,x,\lambda)=\left|\lambda\right|v(x)\left|g(x)\right|$. Note first that for any $x\in U$ and $\lambda\in\C$, the function $R_{x,\lambda}\colon B_{\Hv(U)}\to [0,\infty[$, given by $R_{x,\lambda}(g)=R(g,x,\lambda)$, is continuous. For every $n\in\mathbb{N}$, $\lambda_1,\ldots,\lambda_n\in\mathbb{C}$ and $x_1,\ldots,x_n\in U$, we have  
\begin{align*}
\left(\sum_{i=1}^n S(f,x_i,\lambda_i)^p\right)^{\frac{1}{p}}
&=\left(\sum_{i=1}^n \left|\lambda_i\right|^p v(x_i)^p\left\|f(x_i)\right\|^p\right)^{\frac{1}{p}}\\
&\leq\pi^{\Hv}_p(f)\sup_{g\in B_{\Hv(U)}}\left(\sum_{i=1}^n\left|\lambda_i\right|^p v(x_i)^p\left|g(x_i)\right|^p\right)^{\frac{1}{p}}\\
&=\pi^{\Hv}_p(f)\sup_{g\in B_{\Hv(U)}}\left(\sum_{i=1}^n R(g,x_i,\lambda_i)^p\right)^{\frac{1}{p}},
\end{align*}
and therefore $f$ is $RS$-abstract $p$-summing. Hence, by applying \cite[Theorem 3.1]{PelSan-11}, there is a measure $\mu\in\P(B_{\Hv(U)})$ such that 
$$
S(f,x,\lambda)\leq \pi^{\Hv}_p(f)\left(\int_{B_{\Hv(U)}}R(g,x,\lambda)^p\ d\mu(g)\right)^{\frac{1}{p}}
$$
for all $x\in U$ and $\lambda\in\C$, and therefore 
$$
\left\|f(x)\right\|\leq \pi^{\Hv}_p(f)\left(\int_{B_{\Hv(U)}}\left\vert g(x)\right\vert^{p}d\mu(g)\right)^{\frac{1}{p}}
$$
for all $x\in U$. 

$(ii)\Rightarrow(i)$: Given $n\in\mathbb{N}$, $\lambda_1,\ldots,\lambda_n\in\mathbb{C}$ and $x_1,\ldots,x_n\in U$, we have
\begin{align*}
\sum_{i=1}^n\left|\lambda_i\right|^pv(x_i)^p\left\|f(x_i)\right\|^p
&\leq C^p \sum_{i=1}^n\int_{B_{\Hv(U)}}\left|\lambda_i\right|^p v(x_i)^p\left|g(x_i)\right|^{p}d\mu(g)\\
&= C^p \int_{B_{\Hv(U)}}\left(\sum_{i=1}^n\left|\lambda_i\right|^pv(x_i)^p\left|g(x_i)\right|^{p}\right)d\mu(g)\\
&\leq C^p \sup_{g\in B_{\Hv(U)}}\left(\sum_{i=1}^n\left|\lambda_i\right|^pv(x_i)^p\left|g(x_i)\right|^{p}\right).
\end{align*}
Hence $f\in\Pi^{\Hv}_p(U,F)$ with $\pi^{\Hv}_p(f)\leq C$, and so $\pi^{\Hv}_p(f)\leq\inf\{C\geq 0\colon C \text{ satisfies } (ii)\}$.
\end{proof}


\subsection{Factorization through spaces of continuous functions}

Our initial objetive in this section was to characterize the class of all $p$-summing weighted holomorphic mappings in terms of a factorization through a $L_p(\mu)$-space and a $C(K)$-space (that is, the so-called Pietsch factorization, see \cite[Theorem 3]{Pie-67}, also \cite[Theorem 2.13]{DisJarTon-95}). However we may only get a factorization of such mappings through a space of continuous functions $C(K)$.

Towards this end, consider the map $\iota_U\colon U\to C(B_{\Hv(U)})$ defined by  
$$
\iota_U(x)(g)=g(x)\qquad (x\in U,\; g\in B_{\Hv(U)}),
$$
the inclusion operator $j_p\colon C(B_{\Hv(U)})\to L_p(\mu)$ for some $\mu\in\P(B_{\Hv(U)})$ and $1\leq p\leq\infty$, and the isometric linear embedding $\kappa_F\colon F\to \ell_\infty(B_{F^*})$ given by 
$$
\left\langle \kappa_F(y),y^*\right\rangle=y^*(y)\qquad (y^*\in B_{F^*},\; y\in F).
$$
Note that $\iota_U\in\Hv(U,C(B_{\Hv(U)}))$ with $\|\iota_U\|_v\leq 1$ (see Theorem \ref{t0} (iii)). 


\begin{theorem}\label{Pietsch2} 
Let $1\leq p<\infty$ and $f\in\Hv(U,F)$. Consider the following statements:
\begin{enumerate}
\item[(i)] There are a mapping $h_0\in\Hv(U,C(B_{\Hv(U)}))$, a measure $\mu_0\in\P(B_{\Hv(U)})$ and an operator $S_0\in\L(L_p(\mu_0),\ell_\infty(B_{F^*}))$ such that the following diagram commutes:
$$			
\xymatrix{C(B_{\Hv(U)}) \ar[rr]^{j_p} && L_p(\mu_0) \ar[d]^{S_0} \\
				U\ar[u]^{h_0} \ar[r]^{f} & F \ar[r]^{\kappa_F} & \ell_\infty(B_{F^*})}
$$
\item[(ii)] $f$ is $p$-summing weighted holomorphic.
\item[(iii)] There are an operator $S\in\L(C(B_{\Hv(U)}),\ell_\infty(B_{F^*}))$ and a mapping $h\in\Hv(U,C(B_{\Hv(U)}))$ such that the following diagram is commutative:
$$			
\xymatrix{C(B_{\Hv(U)}) \ar[r]^{S} & \ell_\infty(B_{F^*}) \\
				U\ar[u]^{h} \ar[r]^{f} & F\ar[u]^{\kappa_F}}
$$
\end{enumerate}
Then $(i)\Rightarrow(ii)\Rightarrow(iii)$. In such a case, 
$$
\inf\left\{\left\|S\right\|\|h\|_v\right\}\leq\pi^{\Hv}_p(f)\leq\inf\left\{\left\|S_0\right\|\|h_0\|_v\right\},
$$
where the infimum is taken over all factorizations of $\kappa_F\circ f$ as in $(iii)$ and $(i)$, respectively.
\end{theorem}

\begin{proof}
$(i)\Rightarrow(ii)$: We can write $\kappa_F\circ f=S_0\circ j_p\circ h_0$ as in $(i)$. Since $S_0\circ j_p\in\Pi_p(C(B_{\Hv(U)}),\ell_\infty(B_{F^*}))$ with $\pi_p(S_0\circ j_p)\leq \left\|S_0\right\|\pi_p(j_p)=\left\|S_0\right\|$ by \cite[Theorem 2.4 and Examples 2.9]{DisJarTon-95}, then $\kappa_F\circ f\in\Pi^{\Hv}_p(U,\ell_\infty(B_{F^*}))$ with $\pi^{\Hv}_p(\kappa_F\circ f)\leq\pi_p(S_0\circ j_p)\|h_0\|_v$ by Corollary \ref{new corollary}. By Proposition \ref{ideal summing}, it follows that $f\in\Pi^{\Hv}_p(U,F)$ with $\pi^{\Hv}_p(f)=\pi^{\Hv}_p(\kappa_F\circ f)$, and thus $\pi^{\Hv}_p(f)\leq\left\|S_0\right\|\|h_0\|_v$. Taking the infimum over all factorizations of $\kappa_F\circ f$, one has $\pi^{\Hv}_p(f)\leq\inf\left\{\left\|S_0\right\|\|h_0\|_v\right\}$. 

$(ii) \Rightarrow (iii)$: If $f\in\Pi^{\Hv}_p(U,F)$, then Theorem \ref{Pietsch} yields a measure $\mu\in\P(B_{\Hv(U)})$ such that 
$$
\left\|f(x)\right\|\leq \pi^{\Hv}_p(f)\left(\int_{B_{\Hv(U)}}\left\vert g(x)\right\vert^{p}d\mu(g)\right)^{\frac{1}{p}}\qquad (x\in U).
$$
Take the linear subspace $M:=\lin((v\iota_U)(U))\subseteq C(B_{\Hv(U)})$ and let $y\in M$. Hence $y=\sum_{i=1}^n\alpha_iv(x_i)\iota_U(x_i)$ for some $n\in\mathbb{N}$, $(\alpha_i)_{i=1}^n\in\mathbb{C}^n$ and $(x_i)_{i=1}^n\in U^n$. Define the map $S_1\colon M\to\ell_\infty(B_{F^*})$ by 
$$
S_1(y)=\sum_{i=1}^n\alpha_iv(x_i)\kappa_F(f(x_i)).
$$
We claim that $S_1$ is well defined. Indeed, suppose $\sum_{i=1}^n\alpha_iv(x_i)\iota_U(x_i)=\sum_{j=1}^m\beta_jv(z_j)\iota_U(z_j)$ for some $m\in\mathbb{N}$, $(\beta_j)_{j=1}^m\in\mathbb{C}^m$ and $(z_j)_{j=1}^m\in U^m$. It easily follows that   
$$
\sum_{i=1}^n\alpha_iv(x_i)\delta_{x_i}=\sum_{j=1}^m\beta_jv(z_j)\delta_{z_j},
$$ 
and we conclude that 
\begin{align*}
S_1\left(\sum_{i=1}^n\alpha_iv(x_i)\iota_U(x_i)\right)&=\sum_{i=1}^n\alpha_iv(x_i)\kappa_F(f(x_i))=(\kappa_F\circ T_f)\left(\sum_{i=1}^n\alpha_iv(x_i)\delta_{x_i}\right)\\
&=(\kappa_F\circ T_f)\left(\sum_{j=1}^m\beta_jv(z_j)\delta_{z_j}\right)=\sum_{j=1}^m\beta_jv(z_j)\kappa_F(f(z_j))=S_1\left(\sum_{j=1}^m\beta_jv(z_j)\iota_U(z_j)\right),
\end{align*}
as required. Note that 
$$
\left\|\iota_U(x_i)\right\|_{\infty}=\sup_{g\in B_{\Hv(U)}}\left|g(x_i)\right|\leq \frac{1}{v(x_i)}
$$
for all $i\in\{1,\ldots,n\}$, hence     
$$
\left\|y\right\|_{\infty}
\leq \sum_{i=1}^n\left|\alpha_i\right|v(x_i)\left\|\iota_U(x_i)\right\|_{\infty}
\leq \sum_{i=1}^n\left|\alpha_i\right|,
$$
and therefore 
$$
\left\|y\right\|_{\infty}\leq\pi(y):=\inf\left\{\sum_{i=1}^n\left|\alpha_i\right|\right\},
$$
where the infimum is taken over all $n\in\N$, $(\alpha_i)_{i=1}^n\in\mathbb{C}^n$ and $(x_i)_{i=1}^n\in U^n$ such that $y=\sum_{i=1}^n\alpha_iv(x_i)\iota_U(x_i)$. It is easy to show that $\pi$ defines a norm on $M$. We now prove that $\|y\|_{\infty}=\pi(y)$. Otherwise, suppose that $\left\|y\right\|_{\infty}<\pi(y)$. Take the set 
$$
B=\{z\in M\colon \pi(z)\leq \left\|y\right\|_{\infty}\}.
$$
Clearly, $B$ is a closed convex subset of $(M,\pi)$. An application of Hahn--Banach Separation Theorem to $B$ and $\{y\}$ provides a functional $\eta\in(M,\pi)^*$ with $\left\|\eta\right\|=1$ such that 
$$
\left\|y\right\|_{\infty}=\sup\{\mathrm{Re}(\eta(z))\colon z\in B\}<\mathrm{Re}(\eta(y)).  
$$
Define the function $k_\eta\colon U\to\C$ by 
$$
k_\eta(x)=\eta(\iota_U(x))\qquad (x\in U).
$$
Clearly, $k_\eta\in\Hv(U)$ with $\left\|k_\eta\right\|_v\leq 1$ since $k_\eta\in\H(U)$ and   
$$
v(x)|k_\eta(x)|=|\eta(v(x)\iota_U(x))|\leq\left\|\eta\right\|\pi(v(x)\iota_U(x))\leq 1
$$
for all $x\in U$. Note that 
\begin{align*}
\eta(y)&=\eta\left(\sum_{i=1}^n\alpha_iv(x_i)\iota_U(x_i)\right)=\sum_{i=1}^n\alpha_iv(x_i)\eta(\iota_U(x_i))\\
&=\sum_{i=1}^n\alpha_iv(x_i)k_\eta(x_i)=\sum_{i=1}^n\alpha_iv(x_i)\iota_U(x_i)(k_\eta)=y(k_\eta).
\end{align*}
Then we arrive at the contradiction:
$$
\left\|y\right\|_{\infty}
\geq \left|y(k_\eta)\right|
=\left|\eta(y)\right|
\geq\mathrm{Re}(\eta(y)),
$$
and this proves that $\pi(y)\leq\|y\|_{\infty}$.
	
Clearly, $S_1$ is linear, and a simple proof shows that   
\begin{align*}
\left\|S_1(y)\right\|_{\infty}
&=\left\|\sum_{i=1}^n\alpha_iv(x_i)\kappa_F(f(x_i))\right\|_{\infty}
\leq \sum_{i=1}^n \left|\alpha_i\right|v(x_i)\left\|f(x_i)\right\|\\
&\leq\pi^{\Hv}_p(f)\sum_{i=1}^n \left|\alpha_i\right|\left(\int_{B_{\Hv(U)}}v(x_i)^p\left|g(x_i)\right|^pd\mu(g)\right)^{\frac{1}{p}}
\leq \pi^{\Hv}_p(f)\sum_{i=1}^n\left|\alpha_i\right|.
\end{align*}
Taking the infimum over all representations of $y$, we deduce that $\left\|S_1(y)\right\|_{\infty}\leq\pi^{\Hv}_p(f)\pi(y)$ and thus $\left\|S_1(y)\right\|_{\infty}\leq\pi^{\Hv}_p(f)\left\|y\right\|_{\infty}$. Therefore we have that $S_1\in\L(M,\ell_\infty(B_{F^*}))$ with $\left\|S_1\right\|\leq\pi^{\Hv}_p(f)$. Since the Banach space $\ell_\infty(B_{F^*})$ is injective (see \cite[p. 45]{DisJarTon-95}), there is an operator  $S\in\L(C(B_{\Hv(U)}),\ell_\infty(B_{F^*}))$ such that $\left.S\right|_{M}=S_1$ with $\left\|S\right\|=\left\|S_1\right\|$. Hence we have 
$$
\kappa_F\circ f=S_1\circ\iota_U=S\circ\iota_U,
$$
and thus $\left\|S\right\|\|\iota_U\|_v\leq\pi^{\Hv}_p(f)$. The proof is completed by taking $h:=\iota_U$.
\end{proof}

\subsection{Inclusion and coincidence}

The next result establishes some inclusion and coincidence relations between classes of $p$-summing weighted holomorphic mappings in terms of $p$.

\begin{proposition}\label{prop-1}
If $1\leq p<q\leq\infty$, then 
$$
(\Pi^{\Hv}_p(U,F),\pi^{\Hv}_p)\leq (\Pi^{\Hv}_q(U,F),\pi^{\Hv}_q)\leq (\Pi^{\Hv}_\infty(U,F),\pi^{\Hv}_\infty)=(\Hv(U,F),\|\cdot\|_v).
$$
\end{proposition}

\begin{proof}
Let $n\in\mathbb{N}$, $\lambda_1,\ldots,\lambda_n\in\mathbb{C}$ and $x_1,\ldots,x_n\in U$. We first prove the coincidence. Let $f\in\Pi^{\Hv}_\infty(U,F)$. For all $x\in U$, we have
$$
v(x)\|f(x)\|\leq\pi_\infty^{\Hv}(f)\sup_{g\in B_{\Hv(U)}}v(x)\left|g(x)\right|\leq\pi_\infty^{\Hv}(f).
$$
Hence $f\in\Hv(U,F)$ with $\left\|f\right\|_v\leq\pi^{\Hv}_\infty(f)$. Conversely, given $f\in\Hv(U,F)$, Theorem \ref{t0} provides 
\begin{align*}
\left|\lambda_i\right|v(x_i)\left\|f(x_i)\right\|
&=|\lambda_i|v(x_i)\|T_f(\delta_{x_i})\|\leq \|T_f\|\left|\lambda_i\right|v(x_i)\|\delta_{x_i}\|\\ 
&=\|f\|_v|\lambda_i|v(x_i)\sup_{g\in B_{\Hv(U)}}\left|g(x_i)\right|\\
&=\|f\|_v\sup_{g\in B_{\Hv(U)}}|\lambda_i|v(x_i)\left|g(x_i)\right|
\end{align*}
for all $i\in\{1,\ldots,n\}$, hence    
$$
\max_{1\leq i\leq n}\left|\lambda_i\right|v(x_i)\left\|f(x_i)\right\|\leq \|f\|_v\sup_{g\in B_{\Hv(U)}}\left(\max_{1\leq i\leq n}\left|\lambda_i\right|v(x_i)\left|g(x_i)\right|\right),
$$
and so $f\in\Pi^{\Hv}_\infty(U,F)$ with $\pi^{\Hv}_\infty(f)\leq \|f\|_v$.

We now prove the first inequality. Let $f\in\Pi^{\Hv}_p(U,F)$. Assume first $q<\infty$. Taking 
$$
\beta_i=\left|\lambda_i\right|^{\frac{q}{p}}v(x_i)^{\frac{q}{p}-1}\left\|f(x_i)\right\|^{\frac{q}{p}-1}\qquad (i=1,\ldots,n),
$$
we have 
\begin{align*}
\left(\sum_{i=1}^n\left|\lambda_i\right|^q v(x_i)^q\left\|f(x_i)\right\|^q\right)^{\frac{1}{p}}
&=\left(\sum_{i=1}^n\left|\beta_i\right|^p v(x_i)^p\left\|f(x_i)\right\|^p\right)^{\frac{1}{p}}\\
&\leq \pi^{\Hv}_p(f) \sup_{g\in B_{\Hv(U)}}\left(\sum_{i=1}^n\left|\beta_i\right|^p v(x_i)^p\left|g(x_i)\right|^p\right)^{\frac{1}{p}}.
\end{align*}
Since $q/p>1$ and $(q/p)^*=q/(q-p)$, H\"older inequality yields
\begin{align*}
\sup_{g\in B_{\Hv(U)}}\left(\sum_{i=1}^n\left|\beta_i\right|^pv(x_i)^p\left|g(x_i)\right|^p\right)^{\frac{1}{p}}
&=\sup_{g\in B_{\Hv(U)}}\left(\sum_{i=1}^n\left(\left|\lambda_i\right|v(x_i)\left\|f(x_i)\right\|\right)^{q-p}\left(\left|\lambda_i\right|v(x_i)\left|g(x_i)\right|\right)^p\right)^{\frac{1}{p}}\\
&\leq\left(\sum_{i=1}^n\left|\lambda_i\right|^qv(x_i)^q\left\|f(x_i)\right\|^q\right)^{\frac{1}{p}-\frac{1}{q}}\sup_{g\in B_{\Hv(U)}}\left(\sum_{i=1}^n\left|\lambda_i\right|^qv(x_i)^q\left|g(x_i)\right|^q\right)^{\frac{1}{q}},
\end{align*}
and thus we obtain 
$$
\left(\sum_{i=1}^n\left|\lambda_i\right|^qv(x_i)^q\left\|f(x_i)\right\|^q\right)^{\frac{1}{q}}
\leq \pi^{\Hv}_p(f) \sup_{g\in B_{\Hv(U)}}\left(\sum_{i=1}^n\left|\lambda_i\right|^qv(x_i)^q\left|g(x_i)\right|^q\right)^{\frac{1}{q}}.
$$
This shows that $f\in\Pi^{\Hv}_q(U,F)$ with $\pi^{\Hv}_q(f)\leq\pi^{\Hv}_p(f)$ if $q<\infty$. For $q=\infty$, the same follows from the coincidence and Remark \ref{rem-p-summing weighted holomorphic}.
\end{proof}

\subsection{Maurey extrapolation}

We apply Pietsch domination of $p$-summing weighted holomorphic mappings to give a weighted holomorphic version of Maurey Extrapolation Theorem \cite{Mau-74}. The proof of the next result is based on that of \cite[Theorem 3.17]{DisJarTon-95}. 

\begin{theorem}\label{Pietsch3} 
Let $1<p<q<\infty$ and assume $\Pi^{\Hv}_q(U,\ell_q)=\Pi^{\Hv}_p(U,\ell_q)$. Then $\Pi^{\Hv}_q(U,F)=\Pi^{\Hv}_1(U,F)$ for every complex Banach space $F$. 
\end{theorem}

\begin{proof}
For each $\mu\in\P(B_{\Hv(U)})$, it is clear that $h_\mu:=j_{\infty}\circ\iota_U\in\Hv(U,L_\infty(\mu))$ with $\|h_\mu\|_v\leq 1$. Let $\iota_q\colon L_\infty(\mu)\to L_q(\mu)$ be the formal inclusion operator. Note that $\iota_q\in\Pi_q(L_\infty(\mu),L_q(\mu))$ with $\pi_p(\iota_q)=1$ (see \cite[Examples 2.9]{DisJarTon-95}). Then $\iota_q\circ h_\mu\in\Pi^{\Hv}_q(U,L_q(\mu))$ with $\pi_q^{\Hv}(\iota_q\circ h_\mu)\leq 1$ by Corollary \ref{new corollary}.  

Since $\Pi^{\Hv}_q(U,\ell_q)=\Pi^{\Hv}_p(U,\ell_q)$ and $\pi^{\Hv}_q\leq\pi^{\Hv}_p$ on $\Pi^{\Hv}_p(U,\ell_q)$ by Proposition \ref{prop-1}, the Closed Graph Theorem yields a constant $C\geq0$ such that $\pi_{p}^{\Hv}(f)\leq C\pi_{q}^{\Hv}(f)$ for all $f\in\Pi^{\Hv}_q(U,\ell_q)$. Since $L_q(\mu)$ is an $\mathcal{L}_{q,\lambda}$-space for each $\lambda>1$, we can assure that given $n\in\N$ and $x_1,\ldots,x_n\in U$, the subspace 
$$
G_0=\lin\left(\left\{\iota_q(h_\mu(x_1)),\ldots,\iota_q(h_\mu(x_n))\right\}\right)\subseteq L_q(\mu)
$$ 
embeds $\lambda$-isomorphically into $\ell_q$, that is, $G_0$ is contained in a subspace $G\subseteq L_q(\mu)$ for which there exists an isomorphism $T\colon G\to\ell_{q}$ with $\left\|T\right\|||T^{-1}||<\lambda$. 

Since $T\circ \iota_q\circ h_\mu\in\Pi^{\Hv}_q(U,\ell_q)=\Pi^{\Hv}_p(U,\ell_q)$, we have 
\begin{align*}
&\left(\sum_{i=1}^n\left|\lambda_i\right|^pv(x_i)^p\left\|(\iota_q\circ h_\mu)(x_i)\right\|_{L_q(\mu)}^p\right)^{\frac{1}{p}}\\
&\leq \left\|T^{-1}\right\|\left(\sum_{i=1}^n\left|\lambda_i\right|^p v(x_i)^p\left\|T(\iota_q(h_\mu(x_i)))\right\|_{\ell_{q}}^p\right)^{\frac{1}{p}}\\
&\leq \left\|T^{-1}\right\|C\pi_q^{\Hv}(T\circ \iota_q\circ h_\mu)\sup_{g\in B_{\Hv(U)}}\left(\sum_{i=1}^n\left|\lambda_i\right|^pv(x_i)^p\left|g(x_i)\right|^p\right)^{\frac{1}{p}}\\
&\leq C\left\|T^{-1}\right\|\left\|T\right\|\pi_q^{\Hv}(\iota_q\circ h_\mu)\sup_{g\in B_{\Hv(U)}}\left(\sum_{i=1}^n\left|\lambda_i\right|^pv(x_i)^p\left|g(x_i)\right|^p\right)^{\frac{1}{p}},
\end{align*}
therefore $\pi_p^{\Hv}(\iota_q\circ h_\mu)\leq C\lambda$ for all $\lambda>1$, and thus $\pi_p^{\Hv}(\iota_q\circ h_\mu)\leq C$. Now, by Theorem \ref{Pietsch}, there exists a measure $\widehat{\mu}\in\P(B_{\Hv(U)})$ such that
$$
\left\|(\iota_q\circ h_\mu)(x)\right\|_{L_q(\mu)}
\leq C\left(\int_{B_{\Hv(U)}}\left|g(x)\right|^p\ d\widehat{\mu}(g)\right)^{\frac{1}{p}}
=C\left\|(\iota_q\circ h_{\widehat{\mu}})(x)\right\|_{L_p(\widehat{\mu})}
$$
for all $x\in U$. In the last equality, we have used that 
\begin{align*}
(\iota_q\circ h_{\widehat{\mu}})(x)(g)&=\iota_q(h_{\widehat{\mu}}(x))(g)=h_{\widehat{\mu}}(x)(g)\\
&=j_{\infty}(\iota_U(x))(g)=\iota_U(x)(g)=g(x)
\end{align*}
for all $x\in U$ and $g\in B_{\Hv(U)}$. 

Take a complex Banach space $F$ and let $f\in\Pi^{\Hv}_q(U,F)$. In view of Proposition \ref{prop-1}, we only must show that $f\in\Pi^{\Hv}_1(U,F)$. Theorem \ref{Pietsch} provides again a measure $\mu_0\in\P(B_{\Hv(U)})$ such that
$$
\left\|f(x)\right\|\leq\pi_q^{\Hv}(f)\left\|(\iota_q\circ h_{\mu_0})(x)\right\|_{L_q(\mu_0)}
$$
for all $x\in U$. We claim that there is a constant $c\geq0$ and a measure $\lambda_0\in\P(B_{\Hv(U)})$ such that
$$
\left\|(\iota_q\circ h_{\mu_0})(x)\right\|_{L_q(\mu_0)}\leq c\left\|(\iota_q\circ h_{\lambda_0})(x)\right\|_{L_1(\lambda_0)}
$$
for all $x\in U$. Indeed, define $\lambda_0=\sum_{n=0}^\infty(1/2^{n+1})\mu_n\in\P(B_{\Hv(U)})$, where $(\mu_n)_{n\geq 1}$ is the sequence in $\P(B_{\Hv(U)})$ given by $\mu_{n+1}=\widehat{\mu_n}$ for all $n\in\N_0$, where the measure $\widehat{\mu_n}$ is defined using Theorem \ref{Pietsch}. As $1<p<q$, then $\theta:=(q-p)/p(q-1)\in (0,1)$ and $1/p=\theta+(1-\theta)/q$, and applying Littlewood's Inequality we have 
\begin{align*}
&\left\|(\iota_q\circ h_{\mu_n})(x)\right\|_{L_p(\mu_n)}
=\left(\int_{B_{\Hv(U)}}\left|(\iota_q\circ h_{\mu_n})(x)(g)\right|^{p}\ d\mu_n(g)\right)^{\frac{1}{p}}\\
&\leq\left(\int_{B_{\Hv(U)}}\left|(\iota_q\circ h_{\mu_n})(x)(g)\right|\ d\mu_n(g)\right)^{\theta}
\left(\int_{B_{\Hv(U)}}\left|(\iota_q\circ h_{\mu_n})(x)(g)\right|^q\ d\mu_n(g)\right)^{\frac{1-\theta}{q}}\\
&=\left\|(\iota_q\circ h_{\mu_n})(x)\right\|^{\theta}_{L_1(\mu_n)}\left\|(\iota_q\circ h_{\mu_n})(x)\right\|^{1-\theta}_{L_q(\mu_n)}
\end{align*}
for each $n\in\N_0$ and all $x\in U$. Using H\"older's Inequality and the inequality 
\begin{align*}
\sum_{n=0}^\infty\frac{1}{2^{n+1}}\left\|(\iota_q\circ h_{\mu_{n+1}})(x)\right\|_{L_q(\mu_{n+1})}
&\leq\sum_{n=0}^\infty\frac{1}{2^{n+1}}\left\|(\iota_q\circ h_{\mu_{n+1}})(x)\right\|_{L_q(\mu_{n+1})}+\left\|(\iota_q\circ h_{\mu_0})(x)\right\|_{L_q(\mu_0)}\\ 
&=\sum_{n=0}^\infty\frac{1}{2^{n}}\left\|(\iota_q\circ h_{\mu_{n}})(x)\right\|_{L_q(\mu_{n})}
=2\sum_{n=0}^\infty\frac{1}{2^{n+1}}\left\|(\iota_q\circ h_{\mu_{n}})(x)\right\|_{L_q(\mu_{n})},
\end{align*}
we now obtain 
\begin{align*}
&\sum_{n=0}^\infty\frac{1}{2^{n+1}}\left\|(\iota_q\circ h_{\mu_n})(x)\right\|_{L_q(\mu_n)}
\leq C\sum_{n=0}^\infty\frac{1}{2^{n+1}}\left\|(\iota_q\circ h_{\mu_{n+1}})(x)\right\|_{L_p(\mu_{n+1})}\\
&\leq C\sum_{n=0}^\infty\frac{1}{2^{n+1}}\left\|(\iota_q\circ h_{\mu_{n+1}})(x)\right\|^{\theta}_{L_1(\mu_{n+1})}\left\|(\iota_q\circ h_{\mu_{n+1}})(x)\right\|^{1-\theta}_{L_q(\mu_{n+1})}\\
&\leq C\left(\sum_{n=0}^\infty\frac{1}{2^{n+1}}\left\|(\iota_q\circ h_{\mu_{n+1}})(x)\right\|_{L_1(\mu_{n+1})}\right)^{\theta}
\left(\sum_{n=0}^\infty\frac{1}{2^{n+1}}\left\|(\iota_q\circ h_{\mu_{n+1}})(x)\right\|_{L_q(\mu_{n+1})}\right)^{1-\theta}\\
&\leq C\left(\sum_{n=0}^\infty\frac{1}{2^{n+1}}\left\|(\iota_q\circ h_{\mu_{n+1}})(x)\right\|_{L_1(\mu_{n+1})}\right)^{\theta}
\left(2\sum_{n=0}^\infty\frac{1}{2^{n+1}}\left\|(\iota_q\circ h_{\mu_{n}})(x)\right\|_{L_q(\mu_{n})}\right)^{1-\theta}
\end{align*}
for all $x\in U$, and we deduce that
\begin{align*}
\sum_{n=0}^\infty\frac{1}{2^{n+1}}\left\|(\iota_q\circ h_{\mu_n})(x)\right\|_{L_q(\mu_n)}
&\leq C^{\frac{1}{\theta}}2^{\frac{1-\theta}{\theta}}\left(\sum_{n=0}^\infty\frac{1}{2^{n+1}}\left\|(\iota_q\circ h_{\mu_{n+1}})(x)\right\|_{L_1(\mu_{n+1})}\right)\\
&\leq C^{\frac{1}{\theta}}2^{\frac{1-\theta}{\theta}}2\left(\sum_{n=0}^\infty\frac{1}{2^{n+1}}\left\|(\iota_q\circ h_{\mu_{n}})(x)\right\|_{L_1(\mu_{n})}\right)\\
&=(2C)^{\frac{1}{\theta}}\left\|(\iota_q\circ h_{\lambda_0})(x)\right\|_{L_1(\lambda_0)}
\end{align*}
for all $x\in U$. From above, we deduce that 
$$
\frac{1}{2}\left\|(\iota_q\circ h_{\mu_0})(x)\right\|_{L_q(\mu_0)}\leq (2C)^{\frac{1}{\theta}}\left\|(\iota_q\circ h_{\lambda_0})(x)\right\|_{L_1(\lambda_0)}
$$
for all $x\in U$, and this proves our claim taking $c=2(2C)^{\frac{1}{\theta}}$. Thus we can write
\begin{align*}
\left\|f(x)\right\|
&\leq c\pi_q^{\Hv}(f)\left\|(\iota_q\circ h_{\lambda_0})(x)\right\|_{L_1(\lambda_0)}\\
&=c\pi_q^{\Hv}(f)\int_{B_{\Hv(U)}}\left|g(x)\right|\ d\lambda_0(g)
\end{align*}
for all $x\in U$. Hence $f\in\Pi^{\Hv}_1(U,F)$ with $\pi_1^{\Hv}(f)\leq c\pi_q^{\Hv}(f)$ by Theorem \ref{Pietsch}.
\end{proof}


\section{Banach-valued $\Hv$-molecules}\label{3}

Our main objective in this section is to study the duality of the spaces of $p$-summing weighted holomorphic mappings. With this purpose, we will first establish some results and introduce some concepts influenced by the theory of tensor products (see the monograph \cite{Rya-02}). 

\begin{proposition}\label{prop-navidad}
Let $x\in U$ and $y\in F$. Then $v(x)\delta_x\otimes y\colon\Hv(U,F^*)\to\mathbb{C}$ given by  
$$
(v(x)\delta_x\otimes y)(f)=\left\langle v(x)f(x),y\right\rangle\qquad \left(f\in\Hv(U,F^*)\right), 
$$
belongs to $\Hv(U,F^*)^*$ and $\left\|v(x)\delta_x\otimes y\right\|=v(x)\left\|\delta_x\right\|\left\|y\right\|$. 
\end{proposition}

\begin{proof}
Clearly, $v(x)\delta_x\otimes y$ is linear. For any $f\in\Hv(U,F^*)$, we have
\begin{align*}
\left|(v(x)\delta_x\otimes y)(f)\right|&=\left|\left\langle v(x)f(x),y\right\rangle\right|\leq v(x)\left\|f(x)\right\|\left\|y\right\|=v(x)\left\|T_f(\delta_x)\right\|\left\|y\right\|\\
&\leq v(x)\left\|T_f\right\|\left\|\delta_x\right\|\left\|y\right\|=\left\|f\right\|_vv(x)\left\|\delta_x\right\|\left\|y\right\|,
\end{align*}
and thus $v(x)\delta_x\otimes y\in\Hv(U,F^*)^*$ with $\left\|v(x)\delta_x\otimes y\right\|\leq v(x)\left\|\delta_x\right\|\left\|y\right\|$. For the converse inequality, take $y^*\in B_{F ^*}$ such that $\left|\left\langle y^*,y\right\rangle\right|=\left\|y\right\|$ and consider $g_x\cdot y^*\in\Hv(U,F^*)$ (see Theorem \ref{t0}). Since $\left\|g_x\cdot y^*\right\|_v=\left\|g_x\right\|_v\left\|y^*\right\|\leq 1$, we deduce that 
\begin{align*}
&\left\|v(x)\delta_x\otimes y\right\|\geq \left|(v(x)\delta_x\otimes y)(g_x\cdot y^*)\right|=v(x)\left|\left\langle (g_x\cdot y^*)(x),y\right\rangle\right|\\
&=v(x)\left|\left\langle g_x(x)y^*,y\right\rangle\right|=v(x)\left|g_x(x)\right|\left|\left\langle y^*,y\right\rangle\right|=v(x)\left\|\delta_x\right\|\left\|y\right\|.
\end{align*}
\end{proof}

The elements of the following tensor product space could be referred to as $F$-valued $\Hv$-molecules on $U$. 

\begin{definition}
Let $E$ and $F$ be complex Banach spaces, let $U$ be an open subset of $E$ and let $v$ be a weight on $U$. Define the linear space 
$$
\lin((v\Delta_v)(U))\otimes F:=\lin\left\{v(x)\delta_x\otimes y\colon x\in U,\, y\in F\right\}\subseteq \Hv(U,F^*)^*.
$$
\end{definition}

Each element $\gamma\in\lin((v\Delta_v)(U))\otimes F$ can be expressed in a form (not necessarily unique): 
$$
\gamma=\sum_{i=1}^n\lambda_iv(x_i)\delta_{x_i}\otimes y_i,
$$
where $n\in\mathbb{N}$, $\lambda_i\in\mathbb{C}$, $x_i\in U$ and $y_i\in F$ for $i=1,\ldots,n$; and its action as a functional on a mapping $f\in\Hv(U,F^*)$ comes given by    
$$
\gamma(f)=\sum_{i=1}^n\lambda_iv(x_i)\left\langle f(x_i),y_i\right\rangle . 
$$


\subsection{Projective norm}
 
Given two linear spaces $E$ and $F$, the tensor product space $E\otimes F$, equipped with a norm $\alpha$, is usually denoted by $E\otimes_\alpha F$, and the completion of $E\otimes_\alpha F$ by $E\widehat{\otimes}_\alpha F$. Consider the projective norm $\pi$ on $u\in E\otimes F$, defined by 
$$
\pi(u)=\inf\left\{\sum_{i=1}^n\|x_i\|\left\|y_i\right\|\right\},
$$
where the infimum is taken over all such representations of $u$ as $\sum_{i=1}^n x_i\otimes y_i$ for some $n\in\mathbb{N}$, $(x_i)_{i=1}^n\in E^n$ and $(y_i)_{i=1}^n\in F^n$. 

We now show that the projective norm and the operator canonical norm coincide on the space of $F$-valued $\Hv$-molecules on $U$.

\begin{proposition}\label{teo-L} 
Let $\gamma\in\lin((v\Delta_v)(U))\otimes F$. Then $\left\|\gamma\right\|=\pi(\gamma)$, where 
$$
\left\|\gamma\right\|=\sup\left\{\left|\gamma(f)\right|\colon f\in\Hv(U,F^*), \left\|f\right\|_v\leq 1 \right\}
$$
and 
$$
\pi(\gamma)=\inf\left\{\sum_{i=1}^n\left|\lambda_i\right|v(x_i)\left\|\delta_{x_i}\right\|\left\|y_i\right\|\colon \gamma=\sum_{i=1}^n\lambda_iv(x_i)\delta_{x_i}\otimes y_i\right\}.
$$
\end{proposition}

\begin{proof}
Let $\sum_{i=1}^n\lambda_iv(x_i)\delta_{x_i}\otimes y_i$ be a representation of $\gamma$. Since $\gamma\in\Hv(U,F^*)^*$ and
\begin{align*}
\left|\gamma(f)\right|&=\left|\sum_{i=1}^n\lambda_iv(x_i)\left\langle f(x_i),y_i\right\rangle\right|\leq\sum_{i=1}^n\left|\lambda_i\right|v(x_i)\left\|f(x_i)\right\|\left\|y_i\right\|\\
&=\sum_{i=1}^n\left|\lambda_i\right|v(x_i)\left\|T_f(\delta_{x_i})\right\|\left\|y_i\right\|\leq \left\|f\right\|_v\sum_{i=1}^n\left|\lambda_i\right|v(x_i)\left\|\delta_{x_i}\right\|\left\|y_i\right\|
\end{align*}
for all $f\in\Hv(U,F^*)$, we have that $\left\|\gamma\right\|\leq\sum_{i=1}^n\left|\lambda_i\right|v(x_i) \left\|\delta_{x_i}\right\|\left\|y_i\right\|$. Since this holds for each representation of $\gamma$ as above, it is deduced that $\left\|\gamma\right\|\leq\pi(\gamma)$.  

To prove that $\pi(\gamma)\leq\left\|\gamma\right\|$, suppose by contradiction that $\left\|\gamma\right\|<\pi(\gamma)$. Note that $\gamma\neq 0$ and consider the set 
$$
B=\{\mu\in\lin((v\Delta_v)(U))\otimes F\colon \pi(\mu)\leq \left\|\gamma\right\|\}.
$$
Clearly, $B$ is a closed convex subset of $\lin((v\Delta_v)(U))\otimes_\pi F$. Applying the Hahn--Banach Separation Theorem to $B$ and $\{\gamma\}$, we can take a functional $\eta\in(\lin((v\Delta_v)(U))\otimes_\pi F)^*$ with $\left\|\eta\right\|=1$ such that 
$$
\left\|\gamma\right\|=\sup\{\mathrm{Re}(\eta(\mu))\colon\mu\in B\}<\mathrm{Re}(\eta(\gamma)). 
$$
Define $f_\eta\colon U\to F^*$ by 
$$
\langle f_\eta(x),y\rangle=\eta\left(\delta_x\otimes y\right)\qquad (x\in U,\; y\in F).
$$
We now show that $f_\eta$ is holomorphic. By \cite[Exercise 8.D]{Muj-86}, it suffices to prove that for each $y\in F$, the function $f_{\eta,y}\colon U\to\mathbb{C}$, defined by 
$$
f_{\eta,y}(x)=\eta(\delta_x\otimes y)\qquad (x\in U),
$$
is holomorphic. For it, let $a\in U$ and since $\Delta_v\colon U\to\lin((v\Delta_v)(U))$ is holomorphic, there exists $D\Delta_v(a)\in\mathcal{L}(E,\lin((v\Delta_v)(U)))$ such that 
$$
\lim_{x\to a}\frac{\delta_x-\delta_a-D\Delta_v(a)(x-a)}{\left\|x-a\right\|}=0. 
$$
Define the function $T(a)\colon E\to\mathbb{C}$ by 
$$
T(a)(x)=\eta(D\Delta_v(a)(x)\otimes y)\qquad\left(x\in E\right). 
$$
Clearly, $T(a)$ is linear and  
\begin{align*}
\left|T(a)(x)\right|&=\left|\eta(D\Delta_v(a)(x)\otimes y)\right|\leq \left\|\eta\right\|\pi\left(D\Delta_v(a)(x)\otimes y\right)\\
&\leq \left\|D\Delta_v(a)(x)\right\|\left\|y\right\|\leq \left\|D\Delta_v(a)\right\|\left\|x\right\|\left\|y\right\|
\end{align*}
for all $x\in E$, hence $T(a)\in E^*$. Since  
\begin{align*}
f_{\eta,y}(x)-f_{\eta,y}(a)-T(a)(x-a)
&=\eta(\delta_x\otimes y)-\eta(\delta_a\otimes y)-\eta(D\Delta_v(a)(x-a)\otimes y) \\
&=\eta\left((\delta_x-\delta_a-D\Delta_v(a)(x-a))\otimes y\right)
\end{align*}
for all $x\in U$, it follows that the limit 
$$
\lim_{x\to a}\frac{f_{\eta,y}(x)-f_{\eta,y}(a)-T(a)(x-a)}{\left\|x-a\right\|}
=\lim_{x\to a}\eta\left(\frac{\delta_x-\delta_a-D\Delta_v(a)(x-a)}{\left\|x-a\right\|}\otimes y\right)=0.
$$
Thus $f_{\eta,y}$ is holomorphic at $a$ with $Df_{\eta,y}(a)=T(a)$, as required.

Given $x\in U$, we have  
$$
v(x)|\langle f_\eta(x),y\rangle|
=v(x)\left|\eta\left(\delta_x\otimes y\right)\right|\leq v(x)\left\|\eta\right\|\pi(\delta_x\otimes y)
\leq v(x)\left\|\delta_{x}\right\|\left\|y\right\|\leq\left\|y\right\|
$$
for all $y\in F$, and so $v(x)\left\|f_\eta(x)\right\|\leq 1$. Therefore $f_\eta\in\Hv(U,F^*)$ and $\left\|f_\eta\right\|_v\leq 1$. Furthermore, $\mu(f_\eta)=\eta(\mu)$ for all $\mu\in\lin((v\Delta_v)(U))\otimes F$. Then 
$$
\left\|\gamma\right\|\geq|\gamma(f_\eta)|\geq\mathrm{Re}(\gamma(f_\eta))=\mathrm{Re}(\eta(\gamma)),
$$
and we arrive at a contradiction.
\end{proof}


\subsection{$p$-Chevet--Saphar $\Hv$-norms}

The $p$-Chevet--Saphar norms $d_p$ on the tensor product of two Banach spaces $E\otimes F$ are well known (see, for example, \cite[Section 6.2]{Rya-02}). 

Our study of the duality of the spaces of $p$-summing weighted holomorphic mappings requires the introduction of the following $\Hv$-variants of such norms.

\begin{definition}
Let $E$ and $F$ be complex Banach spaces, let $U$ be an open subset of $E$, let $v$ be a weight on $U$ and let $1\leq p\leq \infty$. For $\gamma\in\lin((v\Delta_v)(U))\otimes F$, define $d^{\Hv}_p(\gamma)$ as 
$$
\inf\left\{\left(\sup_{g\in B_{\Hv(U)}}\left(\sum_{i=1}^n\left|\lambda_i\right|^{p^*}v(x_i)^{p^*}\left|g(x_i)\right|^{p^*}\right)^{\frac{1}{p^*}}\right)\left(\sum_{i=1}^n\left\|y_i\right\|^p\right)^{\frac{1}{p}}\right\}
$$
for any $1<p<\infty$, and
\begin{align*}
d^{\Hv}_1(\gamma)&=\inf\left\{\left(\sup_{g\in B_{\Hv(U)}}\left(\max_{1\leq i\leq n}\left|\lambda_i\right|v(x_i)\left|g(x_i)\right|\right)\right)\left(\sum_{i=1}^n\left\|y_i\right\|\right)\right\},\\
d^{\Hv}_\infty(\gamma)&=\inf\left\{\left(\sup_{g\in B_{\Hv(U)}}\left(\sum_{i=1}^n\left|\lambda_i\right|v(x_i)\left|g(x_i)\right|\right)\right)\left(\max_{1\leq i \leq n}\left\|y_i\right\|\right)\right\},
\end{align*}
with the infimum taken over all representations of $\gamma$ as $\sum_{i=1}^n\lambda_iv(x_i)\delta_{x_i}\otimes y_i$. 
\end{definition}

Motivated by the analogue concept for tensor product spaces, we introduce the following.

\begin{definition}\label{DEF}
A norm $\alpha$ on $\lin((v\Delta_v)(U))\otimes F$ is said to be a $\Hv$-reasonable crossnorm if it enjoys the following properties:
\begin{enumerate}
	\item[(i)] $\alpha(v(x)\delta_x\otimes y)=v(x)\left\|\delta_x\right\|\left\|y\right\|$ for all $x\in U$ and $y\in F$,
	\item[(ii)] For $g\in\Hv(U)$ and $y^*\in F^*$, the linear map $g\otimes y^*\colon\lin((v\Delta_v)(U))\otimes F\to \C$, defined by 
	$$
	(g\otimes y^*)(v(x)\delta_x\otimes y)=v(x)g(x)y^*(y),
	$$
	is bounded on $\lin((v\Delta_v)(U))\otimes_\alpha F$ with $\left\|g\otimes y^*\right\|\leq \left\|g\right\|_v\left\|y^*\right\|$.
\end{enumerate}
\end{definition}

\begin{theorem}\label{teo-che-norms}
$d^{\Hv}_p$ is a $\Hv$-reasonable crossnorm on $\lin((v\Delta_v)(U))\otimes F$ for any $1\leq p\leq\infty$.
\end{theorem}

\begin{proof}
We will only prove it for $1<p<\infty$. The other cases follow similarly. 

Let $\gamma\in\lin((v\Delta_v)(U))\otimes F$ and let $\sum_{i=1}^n\lambda_iv(x_i)\delta_{x_i}\otimes y_i$ be a representation of $\gamma$. Clearly, $d^{\Hv}_p(\gamma)\geq 0$. Given $\lambda\in\C$, since $\sum_{i=1}^n (\lambda\lambda_i)v(x_i)\delta_{x_i}\otimes y_i$ is a representation of $\lambda\gamma$, we have 
\begin{align*}
d^{\Hv}_p(\lambda\gamma)
&\leq\left(\sup_{g\in B_{\Hv(U)}}\left(\sum_{i=1}^n\left|\lambda\lambda_i\right|^{p^*}v(x_i)^{p^*}\left|g(x_i)\right|^{p^*}\right)^{\frac{1}{p^*}}\right)\left(\sum_{i=1}^n\left\|y_i\right\|^p\right)^{\frac{1}{p}}\\
&=\left|\lambda\right|\left(\sup_{g\in B_{\Hv(U)}}\left(\sum_{i=1}^n\left|\lambda_i\right|^{p^*}v(x_i)^{p^*}\left|g(x_i)\right|^{p^*}\right)^{\frac{1}{p^*}}\right)\left(\sum_{i=1}^n\left\|y_i\right\|^p\right)^{\frac{1}{p}}.
\end{align*}
If $\lambda=0$, we obtain $d^{\Hv}_p(\lambda\gamma)=0=\left|\lambda\right|d^{\Hv}_p(\gamma)$. For $\lambda\neq 0$, since the preceding inequality holds for every representation of $\gamma$, we deduce that $d^{\Hv}_p(\lambda\gamma)\leq\left|\lambda\right|d^{\Hv}_p(\gamma)$. For the converse inequality, note that $d^{\Hv}_p(\gamma)=d^{\Hv}_p(\lambda^{-1}(\lambda\gamma))\leq |\lambda^{-1}|d^{\Hv}_p(\lambda\gamma)$ by using the proved inequality, thus $\left|\lambda\right|d^{\Hv}_p(\gamma)\leq d^{\Hv}_p(\lambda\gamma)$ and hence $d^{\Hv}_p(\lambda\gamma)=\left|\lambda\right|d^{\Hv}_p(\gamma)$.

To prove the triangular inequality of $d^{\Hv}_p$, let $\gamma_1,\gamma_2\in\lin((v\Delta_v)(U))\otimes F$ and let $\varepsilon>0$. If $\gamma_1=0$ or $\gamma_2=0$, there is nothing to prove. Assume $\gamma_1\neq 0\neq \gamma_2$. We can choose representations 
$$
\gamma_1=\sum_{i=1}^n\lambda_{1,i}v(x_{1,i})\delta_{x_{1,i}}\otimes y_{1,i},\qquad
\gamma_2=\sum_{i=1}^m\lambda_{2,i}v(x_{2,i})\delta_{x_{2,i}}\otimes y_{2,i},
$$
so that
$$
\left(\sup_{g\in B_{\Hv(U)}}\left(\sum_{i=1}^n\left|\lambda_{1,i}\right|^{p^*}v(x_{1,i})^{p^*}\left|g(x_{1,i})\right|^{p^*}\right)^{\frac{1}{p^*}}\right)\left(\sum_{i=1}^n\left\|y_{1,i}\right\|^p\right)^{\frac{1}{p}}
\leq d^{\Hv}_p(\gamma_1)+\varepsilon
$$
and 
$$
\left(\sup_{g\in B_{\Hv(U)}}\left(\sum_{i=1}^m\left|\lambda_{2,i}\right|^{p^*}v(x_{2,i})^{p^*}\left|g(x_{2,i})\right|^{p^*}\right)^{\frac{1}{p^*}}\right)\left(\sum_{i=1}^m\left\|y_{2,i}\right\|^p\right)^{\frac{1}{p}}
\leq d^{\Hv}_p(\gamma_2)+\varepsilon .
$$
Fix arbitrary $r,s\in\mathbb{R}^+$ and define 
\begin{align*}
\lambda_{3,i}v(x_{3,i})\delta_{x_{3,i}}
&=\left\{\begin{array}{lll}
r^{-1}\lambda_{1,i}v(x_{1,i})\delta_{x_{1,i}}&\text{ if } i=1,\ldots,n,\\
s^{-1}\lambda_{2,i-n}v(x_{2,i-n})\delta_{x_{2,i-n}}&\text{ if } i=n+1,\ldots,n+m,
\end{array}\right.\\
y_{3,i}&=\left\{\begin{array}{lll}
ry_{1,i}&\text{ if } i=1,\ldots,n,\\
sy_{2,i-n}&\text{ if } i=n+1,\ldots,n+m.
\end{array}\right.
\end{align*}
It is clear that $\gamma_1+\gamma_2=\sum_{i=1}^{n+m}\lambda_{3,i}v(x_{3,i})\delta_{x_{3,i}}\otimes y_{3,i}$ and thus we have
$$
d^{\Hv}_p(\gamma_1+\gamma_2)
\leq \left(\sup_{g\in B_{\Hv(U)}}\left(\sum_{i=1}^{n+m}\left|\lambda_{3,i}\right|^{p^*}v(x_{3,i})^{p^*}\left|g(x_{3,i})\right|^{p^*}\right)^{\frac{1}{p^*}}\right)\left(\sum_{i=1}^{n+m}\left\|y_{3,i}\right\|^p\right)^{\frac{1}{p}}.
$$
A simple proof gives  
\begin{align*}
&\left(\sup_{g\in B_{\Hv(U)}}\left(\sum_{i=1}^{n+m}\left|\lambda_{3,i}\right|^{p^*}v(x_{3,i})^{p^*}\left|g(x_{3,i})\right|^{p^*}\right)^{\frac{1}{p^*}}\right)^{p^*}\\
&\leq\left(r^{-1}\sup_{g\in B_{\Hv(U)}}\left(\sum_{i=1}^{n}\left|\lambda_{1,i}\right|^{p^*}v(x_{1,i})^{p^*}\left|g(x_{1,i})\right|^{p^*}\right)^{\frac{1}{p^*}}\right)^{p^*}+\left(s^{-1}\sup_{g\in B_{\Hv(U)}}\left(\sum_{i=1}^{m}\left|\lambda_{2,i}\right|^{p^*}v(x_{2,i})^{p^*}\left|g(x_{2,i})\right|^{p^*}\right)^{\frac{1}{p^*}}\right)^{p^*}
\end{align*}
and
$$
\sum_{i=1}^{n+m}\left\|y_{3,i}\right\|^p
=r^{p}\sum_{i=1}^{n}\left\|y_{1,i}\right\|^p+s^p\sum_{i=1}^{m}\left\|y_{2,i}\right\|^p.
$$
Using Young's Inequality, it follows that 
\begin{align*}
d^{\Hv}_p(\gamma_1+\gamma_2)
&\leq\frac{1}{p^*}\left(\sup_{g\in B_{\Hv(U)}}\left(\sum_{i=1}^{n+m}\left|\lambda_{3,i}\right|^{p^*}v(x_{3,i})^{p^*}\left|g(x_{3,i})\right|^{p^*}\right)^{\frac{1}{p^*}}\right)^{p^*}
+\frac{1}{p}\sum_{i=1}^{n+m}\left\|y_{3,i}\right\|^p\\
&\leq\frac{r^{-p^*}}{p^*}\left(\sup_{g\in B_{\Hv(U)}}\left(\sum_{i=1}^{n}\left|\lambda_{1,i}\right|^{p^*}v(x_{1,i})^{p^*}\left|g(x_{1,i})\right|^{p^*}\right)^{\frac{1}{p^*}}\right)^{p^*}
+\frac{r^p}{p}\sum_{i=1}^{n}\left\|y_{1,i}\right\|^p\\
&+\frac{s^{-p^*}}{p^*}\left(\sup_{g\in B_{\Hv(U)}}\left(\sum_{i=1}^{m}\left|\lambda_{2,i}\right|^{p^*}v(x_{2,i})^{p^*}\left|g(x_{2,i})\right|^{p^*}\right)^{\frac{1}{p^*}}\right)^{p^*}
+\frac{s^p}{p}\sum_{i=1}^{m}\left\|y_{2,i}\right\|^p.
\end{align*}
Since $r,s$ were arbitrary in $\R^+$, taking above
\begin{align*}
r&=(d^{\Hv}_p(\gamma_1)+\varepsilon)^{-\frac{1}{p^*}}\left(\sup_{g\in B_{\Hv(U)}}\left(\sum_{i=1}^{n}\left|\lambda_{1,i}\right|^{p^*}v(x_{1,i})^{p^*}\left|g(x_{1,i})\right|^{p^*}\right)^{\frac{1}{p^*}}\right),\\
s&=(d^{\Hv}_p(\gamma_2)+\varepsilon)^{-\frac{1}{p^*}}\left(\sup_{g\in B_{\Hv(U)}}\left(\sum_{i=1}^{m}\left|\lambda_{2,i}\right|^{p^*}v(x_{2,i})^{p^*}\left|g(x_{2,i})\right|^{p^*}\right)^{\frac{1}{p^*}}\right),
\end{align*}
we obtain that $d^{\Hv}_p(\gamma_1+\gamma_2)\leq d^{\Hv}_p(\gamma_1)+d^{\Hv}_p(\gamma_2)+2\varepsilon$, and thus 
$$
d^{\Hv}_p(\gamma_1+\gamma_2)\leq d^{\Hv}_p(\gamma_1)+d^{\Hv}_p(\gamma_2)
$$
by the arbitrariness of $\varepsilon$. Hence $d^{\Hv}_p$ is a seminorm. To prove that it is a norm,  we can assume without loss of generality that the set $(v\Delta _{v})(U)\subseteq\Gv(U)$ is linearly independent. Using the H\"{o}lder's inequality, note first that  
\begin{align*}
&\left\vert \sum_{i=1}^{n}\lambda _{i}v(x_{i})g\left( x_{i}\right) y^*(y_{i})\right\vert 
\leq \sum_{i=1}^{n}\left|\lambda_{i}\right| v(x_{i})\left|g(x_i)\right|\left\vert y^*(y_{i})\right\vert \\
&\leq \left\|g\right\|_v\left\|y^*\right\|\sum_{i=1}^n\left|\lambda_i\right|v(x_i)\left|\left(\frac{g}{\left\|g\right\|_v}\right)(x_i)\right|\left|\left(\frac{y^*}{\left\|y^*\right\|}\right)(y_i)\right|\\
&\leq \left\|g\right\|_v\left\|y^*\right\|\left(\sum_{i=1}^n\left|\lambda_i\right|^{p^*}v(x_i)^{p^*}\left|\left(\frac{g}{\left\|g\right\|_v}\right)(x_i)\right|^{p^*}\right)^{\frac{1}{p^*}}\left(\sum_{i=1}^n\left|\left(\frac{y^*}{\left\|y^*\right\|}\right)(y_i)\right|
^p\right)^{\frac{1}{p}}\\
&\leq \left\|g\right\|_v\left\|y^*\right\|\sup_{h\in B_{\Hv(U)}}\left(\sum_{i=1}^n\left|\lambda_i\right|^{p^*}v(x_i)^{p^*}\left|h(x_i)\right|^{p^*}\right)^{\frac{1}{p^*}}\left(\sum_{i=1}^n\left\|y_i\right\|^p\right)^{\frac{1}{p}}.
\end{align*}
for any $0\neq g\in B_{\Hv(U)}$ and $0\neq y^*\in B_{F^{\ast}}$ (for $g=0$ or $y^*=0$, there is nothing to prove). Note that the quantity $\sum_{i=1}^{n}\lambda _{i}v(x_{i})g\left(x_{i}\right) y^*(y_{i})$ does not depend on the representation of $\gamma $ because
\begin{equation*}
\sum_{i=1}^{n}\lambda _{i}v(x_{i})g\left( x_{i}\right) y^{\ast
}(y_{i})=\left( \sum_{i=1}^{n}\lambda _{i}v(x_{i})\delta _{x_{i}}\otimes
y_{i}\right) \left( g\cdot y^*\right) =\gamma \left( g\cdot y^{\ast
}\right).
\end{equation*}
Passing to the infimum over all representations of $\gamma $, we deduce that 
\begin{equation*}
\left|\sum_{i=1}^{n}\lambda _{i}v(x_{i})g\left( x_{i}\right) y^*(y_{i})\right|\leq d_p^{\Hv}\left( \gamma \right) ,
\end{equation*}
for any $g\in B_{\Hv(U)}$ and $y^*\in B_{F^{\ast
}}$. Now, if $d_p^{\Hv }\left( \gamma
\right) =0$, we have 
\begin{equation*}
\left( \sum_{i=1}^{n}\lambda _{i}v(x_{i})y^*\left( y_{i}\right) \delta_{x_{i}}\right) \left( g\right) =\sum_{i=1}^{n}\lambda _{i}v(x_{i})g\left(x_{i}\right) y^*(y_{i})=0,
\end{equation*}
for any $g\in B_{\Hv(U)}$ and $y^*\in B_{F^{\ast}}$. For each $y^*\in B_{F^*}$, it follows that $\sum_{i=1}^{n}\lambda _{i}v(x_{i})y^*\left( y_{i}\right) \delta_{x_{i}}=0$, and thus $\lambda _{i}y^*(y_i)=0$ for all $i\in\left\{ 1,\ldots,n\right\} $. Therefore, $\lambda _{i}y_{i}=0$ for all $i\in \left\{ 1,\ldots,n\right\}$, since $B_{F^*}$ separates the points of $F$. Consequently, $\sum_{i=1}^{n}\lambda _{i}v(x_{i})\delta _{x_{i}}\otimes y_{i}=0$, and thus $\gamma=0$.

Finally, we show that $d^{\Hv}_p$ is a $\Hv$-reasonable crossnorm on the tensor space $\lin((v\Delta_v)(U))\otimes F$. Given $g\in\Hv(U)$ and $y^*\in F^*$,  as before we have 
\begin{align*}
\left|(g\otimes y^*)(\gamma)\right|
&=\left|\sum_{i=1}^n\lambda_i(g\otimes y^*)(v(x_i)\delta_{x_i}\otimes y_i)\right|
=\left|\sum_{i=1}^n\lambda_i v(x_i)g(x_i)y^*(y_i)\right|\\
&\leq \left\|g\right\|_v\left\|y^*\right\|\sup_{h\in B_{\Hv(U)}}\left(\sum_{i=1}^n\left|\lambda_i\right|^{p^*}v(x_i)^{p^*}\left|h(x_i)\right|^{p^*}\right)^{\frac{1}{p^*}}\left(\sum_{i=1}^n\left\|y_i\right\|^p\right)^{\frac{1}{p}}.
\end{align*}
Taking infimum over all the representations of $\gamma$, we deduce that 
$$
\left|(g\otimes y^*)(\gamma)\right|\leq \left\|g\right\|_v\left\|y^*\right\|d^{\Hv}_p(\gamma).
$$
Hence $g\otimes y^*\in (\lin((v\Delta_v)(U))\otimes_{d^{\Hv}_p} F)^*$ with $\left\|g\otimes y^*\right\|\leq \left\|g\right\|_v\left\|y^*\right\|$, and thus $d^{\Hv}_p$ holds Condition $(ii)$ in Definition \ref{DEF}. To prove Condition $(i)$, given $x\in U$ and $y\in F$, we have
$$
d^{\Hv}_p(v(x)\delta_x\otimes y)\leq\left(\sup_{g\in B_{\Hv(U)}}v(x)\left|g(x)\right|\right)\left\|y\right\|=v(x)\|\delta_x\|\|y\|,
$$
and, conversely, taking $y^*\in B_{Y^*}$ so that $|y^*(y)|=||y||$, we conclude that  
\begin{align*}
d^{\Hv}_p(v(x)\delta_x\otimes y)
&\geq\left\|g_x\otimes y^*\right\|d^{\Hv}_p(v(x)\delta_x\otimes y)
\geq\left|(g_x\otimes y^*)(v(x)\delta_x\otimes y)\right|\\
&=\left|v(x)g_x(x)y^*(y)\right|=v(x)\left\|\delta_x\right\|\left\|y\right\|.
\end{align*}
\end{proof}

We now compute $d^{\Hv}_p$ with a simpler formula for $p=1$ and $p=\infty$. In fact, the $1$-Chevet--Saphar $\Hv$-norm coincides with the projective norm.

\begin{proposition}\label{1-nuclear-proj}
For $\gamma\in\lin((v\Delta_v)(U))\otimes F$, we have
$$
d^{\Hv}_1(\gamma)=\inf\left\{\sum_{i=1}^n\left|\lambda_i\right|v(x_i)\left\|\delta_{x_i}\right\|\left\|y_i\right\|\right\}
$$ 
and 
$$
d^{\Hv}_\infty(\gamma)=\inf\left\{\sup_{g\in B_{\Hv(U)}}\left(\sum_{i=1}^n\left|\lambda_i\right|v(x_i)\left|g(x_i)\right|\left\|y_i\right\|\right)\right\},
$$
with the infimum taken over all representations of $\gamma$ as $\sum_{i=1}^n\lambda_iv(x_i)\delta_{x_i}\otimes y_i$. 
\end{proposition}

\begin{proof}
Let $\gamma\in\lin((v\Delta_v)(U))\otimes F$ and let $\sum_{i=1}^n\lambda_iv(x_i)\delta_{x_i}\otimes y_i$ be a representation of $\gamma$. We have 
\begin{align*}
\pi(\gamma)&\leq\sum_{i=1}^n\left|\lambda_i\right|v(x_i)\left\|\delta_{x_i}\right\|\left\|y_i\right\|
=\sum_{i=1}^n\left|\lambda_i\right|v(x_i)\left(\sup_{g\in B_{\Hv(U)}}\left|g(x_i)\right|\right)\left\|y_i\right\|\\
&\leq\sum_{i=1}^n\max_{1\leq i\leq n}\left(\left|\lambda_i\right|v(x_i)\sup_{g\in B_{\Hv(U)}}\left|g(x_i)\right|\right)\left\|y_i\right\|\\
&=\left(\max_{1\leq i\leq n}\left(\left|\lambda_i\right|v(x_i)\sup_{g\in B_{\Hv(U)}}\left|g(x_i)\right|\right)\right)\left(\sum_{i=1}^n\left\|y_i\right\|\right)\\
&=\left(\sup_{g\in B_{\Hv(U)}}\left(\max_{1\leq i\leq n}\left|\lambda_i\right|v(x_i)\left|g(x_i)\right|\right)\right)\left(\sum_{i=1}^n\left\|y_i\right\|\right)
\end{align*}
and therefore $\pi(\gamma)\leq d^{\Hv}_1(\gamma)$. Conversely, since $d^{\Hv}_1$ is a crossnorm, we have 
$$
d^{\Hv}_1(\gamma)
\leq \sum_{i=1}^n\left|\lambda_i\right|v(x_i)d^{\Hv}_1(\delta_{x_i}\otimes y_i)
=\sum_{i=1}^n\left|\lambda_i\right|v(x_i)\left\|\delta_{x_i}\right\|\left\|y_i\right\|,
$$
and thus $d^{\Hv}_1(\gamma)\leq\pi(\gamma)$. 

For the second equality we note that  
$$
\sup_{g\in B_{\Hv(U)}}\left(\sum_{i=1}^n\left|\lambda_i\right|v(x_i)\left|g(x_i)\right|\left\|y_i\right\|\right)
\leq \left(\max_{1\leq i\leq n}\left\|y_i\right\|\right)\sup_{g\in B_{\Hv(U)}}\left(\sum_{i=1}^n\left|\lambda_i\right|v(x_i)\left|g(x_i)\right|\right),
$$
and taking the infimum over all representations of $\gamma$ gives  
$$
\inf\left\{\sup_{g\in B_{\Hv(U)}}\left(\sum_{i=1}^n\left|\lambda_i\right|v(x_i)\left|g(x_i)\right|\left\|y_i\right\|\right)\right\}\leq d^{\Hv}_\infty(\gamma).
$$
Conversely, we can assume without loss of generality that $y_i\neq 0$ for all $i\in\{1,\ldots,n\}$ and since $\gamma=\sum_{i=1}^n\lambda_iv(x_i)\left\|y_i\right\|\delta_{x_i}\otimes (y_i/\left\|y_i\right\|)$, we obtain 
$$
d^{\Hv}_\infty(\gamma)\leq\sup_{g\in B_{\Hv(U)}}\left(\sum_{i=1}^n\left|\lambda_i\right|v(x_i)\left\|y_i\right\|\left|g(x_i)\right|\right),
$$
and taking the infimum over all representations of $\gamma$, we conclude that 
$$
d^{\Hv}_\infty(\gamma)\leq\inf\left\{\sup_{g\in B_{\Hv}}\left(\sum_{i=1}^n\left|\lambda_i\right|v(x_i)\left|g(x_i)\right|\left\|y_i\right\|\right)\right\}.
$$
\end{proof}


\subsection{Duality}\label{4}

Given $p\in [1,\infty]$, we will show that the dual space of $\lin((v\Delta_v)(U))\widehat{\otimes}_{d^{\Hv}_{p^*}} F$ can be canonically identified as the space of $p$-summing weighted holomorphic mappings from $U$ to $F^*$.

\begin{theorem}\label{messi-10}\label{theo-dual-classic-00} 
For $1\leq p\leq\infty$, the space $(\Pi^{\Hv}_p(U,F^*),\pi^{\Hv}_p)$ is isometrically isomorphic to $(\lin((v\Delta_v)(U))\widehat{\otimes}_{d^{\Hv}_{p^*}} F)^*$, via the mapping $\Lambda\colon\Pi^{\Hv}_p(U,F^*)\to(\lin((v\Delta_v)(U))\widehat{\otimes}_{d^{\Hv}_{p^*}} F)^*$ defined by 
$$
\Lambda(f)(\gamma)=\sum_{i=1}^n\lambda_iv(x_i)\left\langle f(x_i),y_i\right\rangle 
$$
for $f\in\Pi^{\Hv}_p(U,F^*)$ and $\gamma=\sum_{i=1}^n\lambda_i v(x_i)\delta_{x_i}\otimes y_i\in\lin((v\Delta_v)(U))\otimes F$. Furthermore, its inverse comes given by 
$$
\left\langle \Lambda^{-1}(\varphi)(x),y\right\rangle=\varphi(\delta_x\otimes y) 
$$
for $\varphi\in(\lin((v\Delta_v)(U))\widehat{\otimes}_{d^{\Hv}_{p^*}} F)^*$, $x\in U$ and $y\in F$.
\end{theorem}

\begin{proof}
We prove it for $1<p<\infty$. The cases $p=1$ and $p=\infty$ can be proved similarly.

Let $f\in\Pi^{\Hv}_p(U,F^*)$ and let $\Lambda_0(f)\colon\lin((v\Delta_v)(U))\otimes F\to\mathbb{C}$ be the linear functional given by 
$$
\Lambda_0(f)(\gamma)=\sum_{i=1}^n\lambda_iv(x_i)\left\langle f(x_i),y_i\right\rangle 
$$
for $\gamma=\sum_{i=1}^n\lambda_iv(x_i)\delta_{x_i}\otimes y_i\in\lin((v\Delta_v)(U))\otimes F$. Note that the functional $\Lambda_0(f)$ belongs to $(\lin((v\Delta_v)(U))\otimes_{d^{\Hv}_{p^*}} F)^*$ with 
$\left\|\Lambda_0(f)\right\|\leq \pi^{\Hv}_p(f)$ since  
\begin{align*}
\left|\Lambda_0(f)(\gamma)\right|
&=\left|\sum_{i=1}^n\lambda_iv(x_i)\left\langle f(x_i), y_i\right\rangle\right|\leq\sum_{i=1}^n\left|\lambda_i\right|v(x_i)\left\|f(x_i)\right\|\left\|y_i\right\|\\
&\leq \left(\sum_{i=1}^n\left|\lambda_i\right|^pv(x_i)^p\left\|f(x_i)\right\|^p\right)^{\frac{1}{p}}\left(\sum_{i=1}^n\left\|y_i\right\|^{p^*}\right)^{\frac{1}{p^*}}\\
&\leq \pi^{\Hv}_p(f)\sup_{g\in B_{\Hv(U)}}\left(\sum_{i=1}^n\left|\lambda_i\right|^pv(x_i)^p\left|g(x_i)\right|^p\right)^{\frac{1}{p}}\left(\sum_{i=1}^n\left\|y_i\right\|^{p^*}\right)^{\frac{1}{p^*}},
\end{align*}
and taking the infimum over all such representations of $\gamma$, we deduce the inequality $\left|\Lambda_0(f)(\gamma)\right|\leq \pi^{\Hv}_p(f)d^{\Hv}_{p^*}(\gamma)$. Since $\gamma$ was arbitrary, then $\Lambda_0(f)$ is continuous on $\lin((v\Delta_v)(U))\otimes_{d^{\Hv}_{p^*}} F$ with $\left\|\Lambda_0(f)\right\|\leq \pi^{\Hv}_p(f)$.

Hence there is a unique complex-valued continuous function $\Lambda(f)$ defined on the tensor product space $\lin((v\Delta_v)(U))\widehat{\otimes}_{d^{\Hv}_{p^*}} F$ that extends $\Lambda_0(f)$. Further, $\Lambda(f)$ is linear and $\left\|\Lambda(f)\right\|=\left\|\Lambda_0(f)\right\|$. Let $\Lambda$ be the mapping from $\Pi^{\Hv}_p(U,F^*)$ into $(\lin((v\Delta_v)(U))\widehat{\otimes}_{d^{\Hv}_{p^*}} F)^*$ so defined. 

Clearly, $\Lambda_0\colon\Pi^{\Hv}_p(U,F^*)\to(\lin((v\Delta_v)(U))\otimes_{d^{\Hv}_{p^*}} F)^*$ is linear. To show its injectivity, if $f\in\Pi^{\Hv}_p(U,F^*)$ and $\Lambda_0(f)=0$, then $v(x)\left\langle f(x),y\right\rangle=\Lambda_0(f)(v(x)\delta_x\otimes y)=0$ for all $x\in U$ and $y\in F$, hence $f(x)=0$ for all $x\in U$ and thus $f=0$. Then $\Lambda$ is also linear and injective. Indeed, let $\phi\in\lin((v\Delta_v)(U))\widehat{\otimes}_{d^{\Hv}_{p^*}} F$ and let $(\gamma_n)_{n\geq 1}$ be a sequence in $\lin((v\Delta_v)(U))\otimes_{d^{\Hv}_{p^*}} F$ such that $d^{\Hv}_{p^*}(\gamma_n-\phi)\to 0$ when $n\to +\infty$. Given $a,b\in\C$ and $f,g\in\Pi^{\Hv}_p(U,F^*)$, an easy calculation shows that
\begin{align*}
\Lambda(af+bg)(\gamma_n)&=\Lambda_0(af+bg)(\gamma_n)=(a\Lambda_0(f)+b\Lambda_0(g))(\gamma_n)\\
&=(a\Lambda(f)+b\Lambda(g))(\gamma_n)
\end{align*}
for all $n\in\N$, and taking limits with $n\to +\infty$, we have that 
$$
\Lambda(af+bg)(\phi)=(a\Lambda(f)+b\Lambda(g))(\phi).
$$ 
Hence $\Lambda$ is linear. For the injectivity of $\Lambda$, note that if $f\in\Pi^{\Hv}_p(U,F^*)$ and $\Lambda(f)=0$, then $\Lambda_0(f)=0$ which implies that $f=0$ by the injectivity of $\Lambda_0$. 

To see that $\Lambda$ is a surjective isometry, let $\varphi\in(\lin((v\Delta_v)(U))\widehat{\otimes}_{d^{\Hv}_{p^*}} F)^*$ and define $f_\varphi\colon U\to F^*$ by 
$$
\langle f_\varphi(x),y\rangle=\varphi(\delta_x\otimes y)\qquad\left(x\in U,\; y\in F\right). 
$$
As in the proof of Proposition \ref{teo-L}, it is similarly proved that $f_\varphi\in\Hv(U,F^*)$ with $\left\|f_\varphi\right\|_v\leq\left\|\varphi\right\|$.

We now prove that $f_\varphi\in\Pi^{\Hv}_p(U,F^*)$. Fix $n\in\mathbb{N}$, $\lambda_1,\ldots,\lambda_n\in\mathbb{C}$ and $x_1,\ldots,x_n\in U$. Let $\varepsilon>0$. For each $i\in\{1,\ldots,n\}$, there exists $y_i\in F$ with $\left\|y_i\right\|\leq 1+\varepsilon$ such that $\left\langle f_\varphi(x_i),y_i\right\rangle=\left\|f_\varphi(x_i)\right\|$. Clearly, the function $T\colon\C^n\to\C$, defined by 
$$
T(t_1,\ldots,t_n)=\sum_{i=1}^n t_i \lambda_i v(x_i)\left\|f_\varphi(x_i)\right\|,\qquad\forall (t_1,\ldots,t_n)\in\C^n,
$$
is linear and continuous on $(\C^n,||\cdot||_{p^*})$ with 
$$
\left\|T\right\|=\left(\sum_{i=1}^n\left|\lambda_i\right|^pv(x_i)^p\left\|f_\varphi(x_i)\right\|^{p}\right)^{\frac{1}{p}}.
$$ 
For any $(t_1,\ldots,t_n)\in\C^n$ with $||(t_1,\ldots,t_n)||_{p^*}\leq 1$, we have
\begin{align*}
\left|T(t_1,\ldots,t_n)\right|
&=\left|\varphi\left(\sum_{i=1}^n t_i\lambda_iv(x_i)\delta_{x_i}\otimes y_i\right)\right|\\
&\leq \left\|\varphi\right\|d^{\Hv}_{p^*}\left(\sum_{i=1}^n\lambda_iv(x_i)\delta_{x_i}\otimes t_iy_i\right)\\
&\leq\left\|\varphi\right\|\left(\sup_{g\in B_{\Hv(U)}}\left(\sum_{i=1}^n\left|\lambda_i\right|^{p}v(x_i)^p\left|g(x_i)\right|^{p}\right)^{\frac{1}{p}}\right)\left(\sum_{i=1}^n\left\|t_iy_i\right\|^{p^*}\right)^{\frac{1}{p^*}}\\
&\leq(1+\varepsilon)\left\|\varphi\right\|\sup_{g\in B_{\Hv(U)}}\left(\sum_{i=1}^n\left|\lambda_i\right|^{p}v(x_i)^p\left|g(x_i)\right|^{p}\right)^{\frac{1}{p}},
\end{align*}
therefore 
$$
\left(\sum_{i=1}^n\left|\lambda_i\right|^pv(x_i)^p\left\|f_\varphi(x_i)\right\|^{p}\right)^{\frac{1}{p}}
\leq(1+\varepsilon)\left\|\varphi\right\|\sup_{g\in B_{\Hv(U)}}\left(\sum_{i=1}^n\left|\lambda_i\right|^{p}v(x_i)^p\left|g(x_i)\right|^{p}\right)^{\frac{1}{p}},
$$
and since $\varepsilon$ was arbitrary, it follows that 
$$
\left(\sum_{i=1}^n\left|\lambda_i\right|^pv(x_i)^p\left\|f_\varphi(x_i)\right\|^{p}\right)^{\frac{1}{p}}
\leq\left\|\varphi\right\|\sup_{g\in B_{\Hv(U)}}\left(\sum_{i=1}^n\left|\lambda_i\right|^{p}v(x_i)^p\left|g(x_i)\right|^{p}\right)^{\frac{1}{p}},
$$
and we conclude that $f_\varphi\in\Pi^{\Hv}_p(U,F^*)$ with $\pi^{\Hv}_p(f_\varphi)\leq\left\|\varphi\right\|$. 

Finally, for any $\gamma=\sum_{i=1}^n \lambda_i v(x_i)\delta_{x_i}\otimes y_i\in\lin((v\Delta_v)(U))\otimes F$, we get 
\begin{align*}
\Lambda(f_\varphi)(\gamma)
&=\sum_{i=1}^n\lambda_iv(x_i)\left\langle f_\varphi(x_i),y_i\right\rangle 
=\sum_{i=1}^n\lambda_iv(x_i)\varphi(\delta_{x_i}\otimes y_i)\\
&=\varphi\left(\sum_{i=1}^n\lambda_iv(x_i)\delta_{x_i}\otimes y_i\right) =\varphi(\gamma). 
\end{align*}
Hence $\Lambda(f_\varphi)=\varphi$ on a dense subspace of $\lin((v\Delta_v)(U))\widehat{\otimes}_{d^{\Hv}_{p^*}} F$ and we conclude that $\Lambda(f_\varphi)=\varphi$. This also justifies the last statement of the theorem. Moreover, $\pi^{\Hv}_p(f_\varphi)\leq\left\|\varphi\right\|=\left\|\Lambda(f_\varphi)\right\|$.
\end{proof}

In light of Theorem \ref{messi-10} and taking into account Propositions \ref{prop-1}, \ref{teo-L} and \ref{1-nuclear-proj}, we can identify the space $\Hv(U,F^*)$ with a dual space.

\begin{corollary}\label{theo-dual-classic} 
The space $(\Hv(U,F^*),\|\cdot\|_v)$ is isometrically isomorphic to $(\lin((v\Delta_v)(U))\widehat{\otimes}_{\|\cdot\|} F)^*$, via $\Lambda\colon\Hv(U,F^*)\to(\lin((v\Delta_v)(U))\widehat{\otimes}_{\|\cdot\|} F)^*$ given by 
$$
\Lambda(f)(\gamma)=\sum_{i=1}^n\lambda_iv(x_i)\left\langle f(x_i),y_i\right\rangle 
$$
for $f\in\Hv(U,F^*)$ and $\gamma=\sum_{i=1}^n\lambda_iv(x_i)\delta_{x_i}\otimes y_i\in\lin((v\Delta_v)(U))\otimes F$, with inverse  
$$
\left\langle \Lambda^{-1}(\varphi)(x),y\right\rangle=\varphi(\delta_x\otimes y) 
$$
for $\varphi\in(\lin((v\Delta_v)(U))\widehat{\otimes}_{\|\cdot\|} F)^*$, $x\in U$ and $y\in F$. $\hfill\qed$
\end{corollary}


\section*{Declarations} 


\noindent\textbf{Funding}. This research was partially supported by Junta de Andaluc\'ia grant FQM194, and by grant PID2021-122126NB-C31 funded by ``ERDF A way of making Europe'' and MCIN/AEI/ 10.13039/501100011033.\\

\noindent\textbf{Conflict of interest}. The authors have no relevant financial or non-financial interests to disclose.\\

\noindent\textbf{Competing interests}. The authors declare no competing interests.


\end{document}